\documentclass[reqno]{amsart}

\usepackage[pagewise]{lineno}

\usepackage{color}

\usepackage{hyperref}

\newtheorem{theorem}{Theorem}
\newtheorem{lemma}[theorem]{Lemma}
\newtheorem{corollary}[theorem]{Corollary}
\newtheorem{definition}[theorem]{Definition}

\newtheorem{remark}[theorem]{Remark}

\newtheorem{example}[theorem]{Example}

\numberwithin{theorem}{section}
\numberwithin{equation}{section}

\begin{document}

\title[Effect of a small bounded noise on hyperbolicity]{The effect of a small bounded noise on the hyperbolicity for autonomous semilinear differential equations}

\author[T. Caraballo]{Tom\'as Caraballo$^1$}
\email{caraball@us.es}
\thanks{$^1$ Departamento de Ecuaciones Diferenciales y An\'alisis Num\'erico, Universidad de Sevilla, Spain.}
\author[A. N. Carvalho]{Alexandre N. Carvalho$^2$}
\thanks{$^2$ Instituto de Ci\^encias Matem\'aticas e de Computa\c c\~ao, Universidade de S\~ ao Paulo, Brazil.}
\email{andcarva@icmc.usp.br}

\author[J. A. Langa]{Jos\'e A. Langa$^1$}
\email{langa@us.es}

\author[A. N. Oliveira-Sousa]{Alexandre N. Oliveira-Sousa$^{2}$}
\email{alexandrenosousa@gmail.com}

\subjclass[2010]{Primary 37B55, 37B99, 34D09, 93D09}

\keywords{nonuniform exponential dichotomy, robustness, permanence of hyperbolic equilibria}
\date{}
\begin{abstract}
	In this work we study permanence of hyperbolicity for autonomous differential equations under nonautonomous random/stochastic perturbations. For the linear case, we study 
	robustness and existence of exponential dichotomies for nonautonomous random dynamical systems. Next, we establish a result on the persistence of hyperbolic equilibria for nonlinear differential equations. We show that for each nonautonomous random perturbation of an autonomous semilinear problem with an hyperbolic equilibrium there exists a bounded \textit{random hyperbolic solution} for the associated nonlinear nonautonomous random dynamical systems. Moreover, we show that these random hyperbolic solutions converge to the autonomous equilibrium.
	As an application, we consider a semilinear differential equation with a small nonautonomous multiplicative white noise, and as an example, 
	we apply the abstract results to a strongly
	damped wave equation.
\end{abstract}

\maketitle

\section{Introduction}
The study of permanence properties in dynamical systems has been widely developed in the past decades (see
\cite{Arnold-Boxler,Arnold-Crauel-1,Arnold-Kloeden,Caraballo-Colucci-Cruz,Caraballo-Kloeden-Schmalfu,Caraballo-Langa,Carvalho-Langa-2,Chow-Leiva-existence-roughness,Chow-Leiva-existence-unbounded,Bixiang-Wang-existence-stability,Bixiang-Wang-existence-upper,Yao-Deng,Yao-Zhang,Zhou-Lu-Zhang-1} and references therein).
Some of these works dealt with \textit{exponential dichotomies}, which corresponds to the notion of 
\textit{hyperbolicity} in the non-autonomous framework and gives, for each time,
a decomposition of the space into two parts, one along which solutions decay exponentially to zero forwards, and another along which solutions decay exponentially to zero backwards.
In several of these papers they proved robustness of exponential dichotomies for deterministic 
nonautonomous dynamical systems \cite{Chow-Leiva-existence-roughness,Chow-Leiva-existence-unbounded,Henry-1}, and 
random dynamical systems \cite{Barreira-Dragicevi-Valls,Zhou-Lu-Zhang-1}. 
\par Our purpose is to establish robustness results of exponential dichotomies for \textit{nonautonomous random dynamical systems}. First, we extend the concept of exponential dichotomies 
to encompass random and nonautonomous dynamical systems and provide conditions to guarantee that it persists under perturbation. Then, we apply these abstracts results to study permanence of hyperbolicity on nonautonomous random semilinear differential equations obtained by small perturbations of an autonomous problem.
\par In this way, we consider a autonomous semilinear problem in a Banach space $X$
\begin{equation}\label{newint-autonomous-differential-equation}
\dot{y}=\mathcal{A}y+f_0(y), \ \  t>0, \ y(0)=y \in X,
\end{equation} 
and nonautonomous random perturbations of it
\begin{equation}\label{newint-eq-nonautonomous-random-perturbation}
\dot{y}=\mathcal{A}y+f_\eta(t,\theta_t\omega,y), \  t>\tau, \ y(\tau)=y_\tau \in X, \ \eta\in (0,1],
\end{equation}  
where $\mathcal{A}$ generates a strongly continuous semigroup $\{e^{\mathcal{A}t}: t\geq 0\}\subset \mathcal{L}(X)$, $\theta_t:\Omega\to \Omega$ is a random flow defined in a probability space $(\Omega,\mathcal{F},\mathbb{P})$. 
We assume that there exist a hyperbolic equilibrium for \eqref{newint-autonomous-differential-equation} $y^*_0$, i.e.,
$y_0^*$ is such that 
$f(y_0^*)=-\mathcal{A}y$, and the linearized problem 
$\dot{z}=\mathcal{A}z+f^\prime(y_0^*)z$ admits an exponential dichotomy. Then, we provide conditions to prove \textit{existence and continuity} of ``hyperbolic equilibria" for \eqref{newint-eq-nonautonomous-random-perturbation}. In fact, we show that for each small perturbation $f_\eta$ of \eqref{newint-autonomous-differential-equation}, there exists a global solution of \eqref{newint-eq-nonautonomous-random-perturbation} $\xi_\eta^*$ that presents an \textit{hyperbolic behavior} (which means that
the linear nonautonomous random dynamical system generated by 
\begin{equation*}
\dot{y}=\mathcal{A}y+D_yf_\eta(t,\theta_t\omega,\xi_\eta^*(t,\theta_t\omega))y, \ t\geq \tau,
\end{equation*}
admits an exponential dichotomy), and that these hyperbolic solutions $\xi_\eta^*$ converges to $y^*_0$, as $\eta\to 0$. 
\par To prove this result on the existence and continuity of ``hyperbolic equilibria" for semilinear differential equations, we need first to guarantee permanence of hyperbolicity
for linear nonautonomous random dynamical systems, which can be interpreted as follows: suppose that $f_\eta(t,\theta_t\omega,\cdot):=\mathcal{B}_\eta(t,\theta_t\omega)\in \mathcal{L}(X)$ is a linear perturbation of $\dot{x}=\mathcal{A}x$, and that the autonomous evolution process generated by $\mathcal{A}$ admits an exponential dichotomy, then we prove the linear co-cycle generated by
\begin{equation}\label{newint-linear-nonautonomous-diffeential-equation}
\dot{x}=\mathcal{A}x+\mathcal{B}_\eta(t,\theta_t\omega)x,
\end{equation} 
admits an exponential dichotomy,
under the assumption that $B(t,\theta_t\omega)$ is uniformly small (on time). 
We prove the robustness result via ``discretization'', i.e., we first prove the robustness for the discrete case, then we provide some ``connecting" results between discrete and continuous dynamical systems, and we use these results to establish the robustness result for continuous nonautonomous random dynamical systems.   
\par As an application of these results, we consider a family of stochastic differential equations with a nonautonomous multiplicative white noise
\begin{equation}\label{eq-newintroduction-stochastic}
dy=\mathcal{A}ydt+f_0(y)dt+\eta\kappa_t y\circ dW_t, \ t\geq \tau, \ y(\tau)=y_\tau\in X,
\end{equation}
where $\eta\in [0,1]$, and the mapping $\mathbb{R}\ni t\mapsto \kappa_t\in \mathbb{R}$ is a real function. We use a formal change of variables to obtain a family of nonautonomous random differential equations like \eqref{newint-eq-nonautonomous-random-perturbation}. As an example, we consider 
a damped wave equation with
a nonautonomous multiplicative white noise
\begin{equation}\label{eq-newint-damped-wave}
\left\{
\begin{array}{l l}
u_{tt}+\beta u_t-\Delta u=f(u)+\eta\kappa_t u\circ \dot{W}, &  x\in D,\\
u(0,x)=h_1(x), \ u_t(0,x)=h_2(x), & x\in D,\\
u(t,x)=0, & x\in \partial D, \  t>0.
\end{array}
\right.
\end{equation}
\par This work was motivated by the following stochastic differential equation, with a small (autonomous) multiplicative white noise,
\begin{equation}\label{new-eq-autonomous-noise}
dy_t=Aydt+\epsilon y\circ dW_t, \ t\geq \tau, \ y(\tau)=y_\tau\in X.
\end{equation}
where $\dot{y}=Ay$ admits an exponential dichotomy, the family $\{W_t:t\in \mathbb{R}\}$ is the standard Wiener process, and
$\epsilon>0$.
Consider a stochastic process
$z(t,\omega)=z^*(\theta_t \omega)$ known as the \textit{Orstein-Uhlenbeck process} (which will be describe later),
where $\omega$ is a parameter in the probability space $(\Omega,\mathcal{F},\mathbb{P})$, and
$\{\theta_t: t\in \mathbb{R}\}$ is a group on $\Omega$.
By a standard procedure, we define 
$v(t,\omega)=e^{-\epsilon z^*(\theta_t\omega)}y(t,\omega)$,
that satisfies
\begin{equation}\label{newintroduction-linear-stochastic}
\dot{v} =Av
+\epsilon z^*(\theta_t\omega)v,
\end{equation}
which is a (autonomous) random differential equation as a perturbation of an autonomous problem. However, for each random parameter $\omega\in \Omega$, our perturbation $B_\epsilon(\theta_t\omega):=\epsilon z^*(\theta_t\omega)$ have some sub-linear growth on $t$, which brings some fundamental difficulties to prove robustness, because we cannot assume that operator $B(\theta_t\omega)$ is uniformly bounded in time. 
Therefore, when 
$\dot{v}=Av$ is hyperbolic it is not possible to guarantee that \eqref{newintroduction-linear-stochastic} admits an exponential dichotomy, in the sense of \cite{Barreira-Dragicevi-Valls,Zhou-Lu-Zhang-1}, for $\epsilon>0$ small enough.
To deal with this issue we propose ``to bound'' the noise with respect to time.
\par Bounded noise is very sensible in real life applications, \cite{Caraballo-Colucci-Cruz,Caraballo-Garrido-Cruz,Carr,Onofrio,Yao-Deng}. For instance,
in \cite{Caraballo-Garrido-Cruz} the authors consider an Ornstein-Uhlenbeck process depending on a control parameter that allows them to ensure
that the noise is inside a bounded interval, which is fixed by practitioners based on experiments. They use such a bounded noise to model perturbations on the input flow of a chemostat. This idea is generalized in \cite{Caraballo-Colucci-Cruz}, where several examples in population dynamics with bounded random fluctuations are given. 
In this work, we ``bound" the noise in the following way:
\par Since the stochastic process $z^*$ satisfies 
$\lim_{t\to \pm\infty}|z^*(\theta_t\omega)/t|=0$, for almost all $\omega$, see \cite{Caraballo-Kloeden-Schmalfu},
it is possible to choose any differentiable positive real function $\kappa$ such that 
there is a random variable $m>0$ such that
\begin{equation*}
m(\omega):=\sup_{t\in \mathbb{R}}\big\{|\kappa_tz^*(\theta_t\omega)|+
|\dot{\kappa}_tz^*(\theta_t\omega)|\big\}<+\infty.
\end{equation*}
Hence, instead of \eqref{new-eq-autonomous-noise}, we consider
\begin{equation}
dy_t=Aydt+\epsilon\kappa_t y\circ dW_t,
\end{equation}\label{new-eq-nonautonomous-noise}
where $\kappa_t$ is a real function that allows us to control the white noise on time. By the change of variables
$v(t,\omega)=e^{-\epsilon\kappa_tz^*(\theta_t\omega)}y(t,\omega)$, we obtain
\begin{equation}\label{new-introduction-linear-random}
\dot{v} =Av
-\epsilon [\dot{\kappa}_t z^*(\theta_t\omega)-\kappa_t z^*(\theta_t\omega)]v,
\end{equation}
which is a linear nonautonomous random differential equation.
Since the perturbation
$B_\epsilon(t,\theta_t\omega):=\epsilon [\dot{\kappa}_t -\kappa_t ]z^*(\theta_t\omega)$ is 
uniformly bounded on time, we can develop a theory of exponential dichotomies to guarantee existence of hyperbolicity for \eqref{new-introduction-linear-random}.
\par Historically, to study permanence of properties under perturbations it is sensible to assume that the perturbation is uniformly bounded on time, see \cite{Arnold-Boxler,Arnold-Crauel-1,Arnold-Kloeden,Barreira-Dragicevi-Valls,Bortolan-Cardoso-etal,Carvalho-Langa-2,Henry-1}. For instance, if $\dot{y}=Ay$ is \textit{hyperbolic}, and $B:\mathbb{R}\to \mathcal{L}(X)$ is uniformly bounded with respect of $t$, then the hyperbolicity persists on $\dot{y}=Ay+B(t)y$, see \cite{Henry-1}.
Also, for a differential equation on a Banach space, driven by a group $\theta_t:\Sigma\to\Sigma$ on a compact Hausdorff set $\Sigma$, $\dot{y}=A(\theta_t\sigma)y$, exhibiting an exponential dichotomy, \cite{Chow-Leiva-existence-roughness} proved that the hyperbolicity persists for $B:\Sigma\to \mathcal{L}(X)$, which is naturally uniformly bounded by the compactness of $\Sigma$, and in \cite{Chow-Leiva-existence-unbounded} they considered an the case that $B$ is a unbounded operator, but also with some uniformly boundedness condition on time. 
\par More recently, they consider the random case, in \cite{Barreira-Dragicevi-Valls,Zhou-Lu-Zhang-1} it was studied stability of tempered exponential dichotomies for random difference equations, $y_{n+1}=A(\theta\omega)y_n$, where $\theta :\Omega\to \Omega$ is a random flow defined on a probability space $(\Omega,\mathcal{F},\mathbb{P})$, but usually in the situation the assumption on the perturbation $B$ are more restricted because has to be exponentially small, typically $\|B(\theta_n\omega)\|\leq \delta(\omega)e^{-\nu|n|}$. Therefore, all these results on robustness assume that the perturbation are uniformly bounded on time. However, as seen before, we cannot assume that even a small noise is satisfy such property. For this reason, to prove permanence of hyperbolicity we have to consider a time dependent scalar function $\kappa_t$ to compensate the 
the growth of the noise in time, \eqref{new-eq-nonautonomous-noise}.
\par Accordingly, we consider a small nonautonomous random perturbation $B:\mathbb{R}\times \Omega\to \mathcal{L}(X)$ of a hyperbolic problem 
$\dot{y}=Ay$. 
Our perturbation $B(t,\omega)$ depends on two parameters, 
the time $t$ of deterministic nature, and another $\omega$ varying in a probability set $(\Omega,\mathcal{F},\mathbb{P})$. This leads us to establish robustness results of exponential dichotomies for \textit{nonautonomous random dynamical systems}, which is a co-cycle $(\varphi,\Theta)$ driven by $\Theta:\mathbb{R}\times\Omega\to \mathbb{R}\times\Omega$. 
Hence, using similar techniques to those in \cite{Zhou-Lu-Zhang-1} (for the discrete case), and \cite{Chow-Leiva-existence-roughness} (for the continuous case), we establish stability results, and applied these results to guarantee existence of exponential dichotomy for  
$\dot{y}=Ay+B(\Theta_t(\tau,\omega))y$, where $B$ is uniformly small (with respect to $t$), and $\dot{x}=Ax$ is hyperbolic.
\par For the nonlinear case, inspired by \cite{Carvalho-Langa-2}, considering 
a small nonautonomous bounded noise as in \eqref{eq-newintroduction-stochastic} we 
provide conditions that allow us to prove existence and continuity of \textit{random hyperbolic solutions} for \eqref{newint-eq-nonautonomous-random-perturbation}. 
In \cite{Carvalho-Langa-2} they study existence and continuity of hyperbolic solutions, invariant manifolds, and provide conditions to prove continuity of pullback attractors for evolution processes. This work was crucial for further development on continuity and stability of attractors in a deterministic scenario
\cite{Bortolan-Cardoso-etal,Bortolan-Carvalho-Langa,Carvalho-Langa-Robinson}.
However, for random dynamical systems the theory is still under development. There are many works on the existence and 
upper semi continuity of random attractors, including
nonautonomous random dynamical systems, see \cite{BATES201432,Caraballo-Langa,Bixiang-Wang-existence,Bixiang-Wang-existence-upper,Yao-Zhang}, and references therein. 
Nevertheless, as it is suggested in \cite{Carvalho-Langa-2,Carvalho-Langa-Robinson}, the lack of theorems on the persistence of random hyperbolic solutions 
makes very difficult to prove results on the lower semi-continuity of random attractors for problems under small random perturbations. The study of these special solutions are the core on the problem of continuity and stability of attractors, and we believe that this paper provides a direction to study continuity and stability of random attractors.
\par For random dynamical systems, Arnold, Boxler, Crauel and Kloeden \cite{Arnold-Boxler,Arnold-Crauel-1,Arnold-Kloeden}, studied the effect of the noise
on hyperbolic dynamical systems. 
In \cite{Arnold-Boxler} they proved existence 
of a stationary solution for a nonlinear random differential equation. 
More recently, in \cite{Bixiang-Wang-existence-stability}, they studied stability of tempered stationary solutions, and tempered attractors for small random perturbations of nonautonomous problems.
However, differently of our work, in none of these works it is possible to conclude that their stationary solutions exhibit hyperbolicity, because the noise presents some growth in time that do not allow us to apply any robustness result. 
\par We organize our paper as follows. In Section \ref{sec-uniform-exponential-dichotomies}, we present a notion of exponential dichotomy, simultaneously, for discrete and continuous nonautonomous random dynamical systems. 
Then, for the discrete case, we prove a robustness result, uniqueness, and continuous dependence of projections associated with the exponential dichotomy for the discrete case in Subsection \ref{subsec-ed-discrete}. In Subsection \ref{subsec-ed-continuous} we establish some theorems to compare exponential dichotomies of discrete and continuous nonautonomous random dynamical systems which allows to conclude the same stability results for the continuous case.
In Section \ref{sec-applications}, we apply the abstract robustness results for
nonautonomous random/stochastic differential equations.
In Subsection \ref{subsec-app-linear}, we show existence of 
exponential dichotomy for \eqref{newint-linear-nonautonomous-diffeential-equation} for the linear case. 
Later, in Subsection \ref{subsection-randomperturbation-of-autonomous-semilinear-problem},
we consider a family of nonautonomous random semilinear differential equations, we establish the existence and continuity of 
\textit{random hyperbolic solutions} for \eqref{newint-eq-nonautonomous-random-perturbation}.
Finally, in Subsection \ref{subsec-app-stochastic}, we show how to apply our result on persistence of hyperbolic equilibria for  
\eqref{eq-newintroduction-stochastic}, and as an example, we consider 
a damped wave equation with
a nonautonomous multiplicative white noise \eqref{eq-newint-damped-wave}.

\section{Exponential dichotomies for Nonautonomous RDS}\label{sec-uniform-exponential-dichotomies}

\par In this section, we introduce the notion of \textit{nonautonomous random dynamical systems} 
in a Banach space $X$. We consider a driving flow $\{\Theta_t\}_{t\in \mathbb{T}}$
over $\mathbb{T}\times\Omega$, where $(\Omega,\mathcal{F},\mathbb{P})$ is a probability space, and
$\mathbb{T}=\mathbb{Z}$ or $\mathbb{T}=\mathbb{R}$.
\begin{definition}
	Let $(\Omega,\mathcal{F},\mathbb{P})$ be a probability space. 
	We say that a family of maps $\{\theta_t:\Omega\rightarrow\Omega: \, t\in \mathbb{T}\}$ is a
	\textbf{random flow} if it satisfies
	\begin{itemize}
		\item $\theta_0=Id_\Omega$;
		\item $\theta_{t+s}=\theta_t\circ \theta_s$, for all $t,s\in \mathbb{T}$;
		\item $\theta_t:\Omega\rightarrow \Omega$ is measurable for all $t\in \mathbb{T}$.
	\end{itemize}
\end{definition}

%

%

\begin{definition}
	Let $\theta:=\{\theta_t:\Omega\rightarrow\Omega: \, t\in \mathbb{T}\}$ be a random flow. 
	Define $\Theta_t(\tau,\omega):=(t+\tau,\theta_t\omega)$ for each $(\tau,\omega)\in \mathbb{T}\times\Omega$, and $t\in \mathbb{T}$.
	We say that a family of maps $\{\varphi(t,\tau,\omega):X\to X; (t,\tau,\omega)\in \mathbb{T}^+\times\mathbb{T}\times\Omega\}$ 
	is a \textbf{nonautonomous random dynamical system (co-cycle)} driven by
	$\Theta$ if
	\begin{enumerate}
		\item the mapping
		$\mathbb{T}^+\times\Omega\times X\ni (t,\omega,x)\mapsto \varphi(t,\tau,\omega)x\in X$
		is measurable
		for each fixed $\tau\in \mathbb{T}$;
		\item
		$\varphi(0,\tau,\omega)=Id_X$,
		for each $(\tau,\omega)\in \mathbb{T}\times\Omega$;
		\item 
		$\varphi(t+s,\tau,\omega)=\varphi(t,\Theta_s(\tau,\omega))\circ\varphi(s,\tau,\omega)$,
		for every $t,s\geq 0$ in 
		$\mathbb{T}$, and $(\tau\omega)\in\mathbb{T}\times \Omega$;
		\item $\varphi(t,\tau,\omega):X\to X$ is a continuous map for each  $(t,\tau,\omega)\in \mathbb{T}^+\times\mathbb{T}\times\Omega$.
	\end{enumerate}
	We usually denote the pair $(\varphi,\Theta)_{(X,\mathbb{T}\times\Omega)}$, or $(\varphi,\Theta)$, to denote the co-cycle $\varphi$ driven by $\Theta$.
\end{definition}

\begin{remark}
	To simplify the notation we will write $\omega_\tau:=(\tau,\omega)\in \mathbb{T}\times \Omega$, and
	$\Theta_t(\omega_\tau):=(\theta_t\omega)_{\tau+t}$. 
\end{remark}

\begin{remark}\label{remark-associated-processes}
	In many parts of this work we will associate a co-cycle 
	$(\varphi,\Theta)$ with a family of evolutions processes. In fact, 
	let $(\varphi,\Theta)$ be a nonautonomous random dynamical system.  
	For each $\omega_p\in \mathbb{T}\times\Omega$ we define the following \textbf{evolution process} 
	\begin{equation*}
	\Phi_{\omega_p}:=\{\varphi_{t,s}(\omega_p):=\varphi(t-s,\Theta_s\omega_p)\,; \, t\geq s \}.
	\end{equation*}
	This means that, for each $\omega_p\in \mathbb{T}\times \Omega$ the family of maps $\Phi_{\omega_p}$ satisfies
	\begin{enumerate}
		\item $\phi_{t,t}(\omega_p)=Id_X$, for each $t\in \mathbb{T}$;
		\item $\phi_{t,s}(\omega_p)\circ \phi_{s,r}=\varphi_{t,r}(\omega_p)$, for each
		$t\geq s\geq r$;
		\item the mapping $\{(t,s)\in \mathbb{T}^2: t\geq s\}\times X\ni (t,s,x)\mapsto \varphi_{t,s}(\omega_p)x\in X$ is continuous. 
	\end{enumerate}
\end{remark}

\par In this paper, we will use, in many of our proofs, the concept of \textit{global solutions} for the evolution process $\Phi_{\omega_p}$, for each $\omega_p\in \mathbb{T}\times \Omega$, associated with a given co-cycle $(\varphi,\Theta)$.
\begin{definition}
	Let $\mathcal{S}=\{S(t,s): t\geq s, t,s\in \mathbb{T}\}$ a evolution process on a Banach space $X$. We say that
	a map $\xi:\mathbb{R}\to X$ is a \textbf{global solution} for $\mathcal{S}$ if
	$S(t,s)\xi(s)=\xi(t)$ for every $t,s\in \mathbb{T}$ with $t\geq s$. 
	\par A global solution $\xi$ is \textbf{backwards bounded }if $\xi(-\infty,0])=\{\xi(t): t\leq 0\}$ is a bounded subset of $X$.
\end{definition}
\par We set $\mathcal{L}(X)$ as the space of all bounded linear maps between $X$ and $X$.
\par Recall the definition of \textit{strongly measurable}:
\begin{definition}
	Let $\Omega$ be a measurable space, and $X$ a Banach space. A map
	$P:\Omega\rightarrow \mathcal{L}(X)$ is said to be \textbf{strongly measurable}
	if for every $x\in X$ the map
	$\Omega\ni\omega\mapsto P(\omega)x\in X$ is measurable.	
\end{definition}
\begin{definition}
	A map $D:\mathbb{T}\times \Omega\to \mathbb{R}$ is said to be \textbf{$\Theta$-invariant} if
	for each $\omega_\tau\in \mathbb{R}\times\Omega$ we have that
	$D(\Theta_t\omega_\tau)=D(\omega_{\tau})$, for every $t\in\mathbb{T}$.
\end{definition}
\par Now, we define the notion of \textit{exponential dichotomy} for linear nonautonomous random dynamical systems.
\begin{definition}
	A nonautonomous random dynamical system $(\varphi,\Theta)$ such that
	$\varphi(t,\tau,\omega)\in \mathcal{L}(X)$, for all 
	$(t,\tau,\omega)\in \mathbb{T}^+\times\mathbb{T}\times\Omega$, is said to admit
	an \textbf{(uniform) exponential dichotomy} if 
	there exists a  $\theta$-invariant subset $\tilde{\Omega}$ of  $\Omega$ with full measure, 
	$\mathbb{P}(\tilde{\Omega})=1$, and a family of projections,
	$\Pi^s:=\{\Pi^s(\omega_\tau): \omega_\tau\in \mathbb{T}\times \tilde{\Omega}\}$
	such that
	\begin{enumerate}
		\item for each $\tau\in \mathbb{T}$ the map $\Pi_\tau^s(\cdot):=\Pi^s(\tau,\cdot):\tilde{\Omega}\to \mathcal{L}(X)$ is strongly measurable;
		\item $\Pi^s(\Theta_t\omega_\tau)\varphi(t,\omega_\tau)=
		\varphi(t,\omega_\tau)\Pi^s(\omega_\tau)$, for every
		$t\in \mathbb{T}^+$ and $\omega_\tau\in \mathbb{T}\times \tilde{\Omega}$;
		\item $\varphi(t,\omega_\tau):R(\Pi^u(\omega_\tau))\to R(\Pi^s(\Theta_t\omega_\tau)) $ is an isomorphism, where
		$\Pi^u_{\tau}:=Id_X-\Pi^s_{\tau}$ for all $\tau\in \mathbb{T}$;
		\item there exist $\Theta$-invariant maps
		$\alpha:\mathbb{T}\times\Omega\to (0,+\infty)$ and $K:\mathbb{T}\times\Omega\to [1,+\infty)$
		such that 
		\begin{eqnarray*}
		\|\varphi(t,\omega_\tau)\Pi^s(\omega_\tau)\|_{\mathcal{L}(X)}
		&\leq&
		K(\omega_\tau)e^{-\alpha(\omega_\tau)t}, \hbox{ for every } t\geq 0;\\
		\|\varphi(t,\omega_\tau)\Pi^u(\omega_\tau)\|_{\mathcal{L}(X)}
		&\leq&
		K(\omega_\tau)e^{\alpha(\omega_\tau)t}, \hbox{ for every } t\leq 0,
		\end{eqnarray*}
		for every $\omega_\tau\in \mathbb{T}\times\tilde{\Omega}$
	\end{enumerate}
	In this case, the function $K$ is called a \textbf{bound} and
	$\alpha$ an \textbf{exponent} for the exponential dichotomy.
	\par We refer to the exponential dichotomy as:
	\textbf{continuous} if $\mathbb{T}:=\mathbb{R}$, and \textbf{discrete} if $\mathbb{T}:=\mathbb{N}$.
\end{definition}

\begin{remark}
	\par If for each $\omega_p$ the map $\mathbb{R}\ni t\to K(\Theta_t\omega_p)$ is not constant we say that 
		$(\varphi,\Theta)$ admits an \textbf{nonuniform (with respect to $t$) exponential dichotomy}. 
		In the special case when the mapping $\mathbb{R}\ni t\to K(\Theta_t\omega_p)$ is \textit{tempered} we say that $(\varphi,\Theta)$ admits a \textbf{tempered exponential dichotomy}
		In this work, we do not deal with the case of nonuniform exponential dichotomies. For more information on this topic we recommend \cite{Barreira-Dragicevi-Valls,Barreira-Valls-Sta,Zhou-Lu-Zhang-1}.
\end{remark}
%

\begin{remark}
	If a co-cycle $(\varphi,\Theta)$ admits an exponential dichotomy, then for each fixed
	$\omega_p\in \mathbb{T}\times \tilde{\Omega}$ the associated evolution process
	$\Phi_{\omega_p}$ also admits it in the sense of Henry \cite[Section 7.6]{Henry-1}.
\end{remark}

\par The interpretation of a nonautonomous random dynamical system as a family of evolution processes with a parameter provides also a special function 
usually called as \textit{Green functions}:
\begin{definition}
	Let $(\varphi,\Theta)$ be a co-cycle which admits 
	an exponential dichotomy with family of projections 
	$\Pi^s$. 
	A \textbf{Green function} associated to $(\varphi,\Theta)$ and family of 
	projection $\Pi^s$ is given by
	\begin{equation*}
	G_{\omega_p}(t,s)= \left\{ 
	\begin{array}{l l} 
	\varphi_{t,s}(\omega_p)\Pi^s(\Theta_s\omega_p), 
	&  \quad \hbox{if } t\geq s,
	\\ -\varphi_{t,s}(\omega_p)\Pi^u(\Theta_s\omega_p), \, 
	& \quad \hbox{if } t<s,
	\end{array} 
	\right.
	\end{equation*} 
	for each $\omega_p$ fixed. 
\end{definition}
%

\subsection{Exponential dichotomy for nonautonomous random co-cycles: discrete case}\label{subsec-ed-discrete}

\quad
\par In this subsection, we study a discrete nonautonomous random dynamical system with exponential dichotomy. In this paper, the purpose of the discrete case is to works as a tool to obtain results for the continuous case of exponential dichotomies and hence, for differential equations. The goal is to presented a summary of results concerning exponential dichotomy for discrete co-cycles driving by flows over non-compact symbols spaces that we are going to need in order to establish robustness results of hyperbolicity for differential equations.
We prove that the property of admitting an discrete exponential dichotomy is stable under perturbation (Theorem \ref{th-discrete-robstness-random-perturbations-of-autonomous-nonuniform-case}), a type of admissibility result (Theorem \ref{th-admissibility-for-pertubed-process}), and uniqueness and continuous dependence of projections (see Corollary \ref{cor-uniqueness-projection-discrete} and Theorem \ref{th-continuity-projection-continuous}, respectively).
\par The techniques used in this subsection are the same introduced by Henry \cite{Henry-1} (for deterministic dynamical systems) and by Zhou \textit{et al.} in \cite{Zhou-Lu-Zhang-1} (for random dynamical systems).

\par A linear discrete nonautonomous random dynamical systems
$(\varphi,\Theta)$
can be associated with
nonautonomous random difference equations. In fact, for each $\omega_p\in \mathbb{Z}\times\Omega$ we study
\begin{equation}\label{equation-linear-homogeneus-nonautonomous-random}
x_{n+1}=A(\Theta_n\omega_p)x_n, \ \  x_n\in X \hbox{ and } n\in \mathbb{Z},
\end{equation}
where $A:\mathbb{T}\times \Omega\rightarrow \mathcal{L}(X)$ 
and 
$\varphi(n,\omega_p):=A(\Theta_{n-1}\omega_p)\circ \cdots \circ A(\omega_p)$ for $n>0$ and $\varphi(0,\omega_p)=Id_X$.

\par We prove existence of solutions for the non-homogeneous problem
\begin{equation}\label{eq-random-perturbed-nonhomogen-linear}
x_{n+1}=A(\Theta_n\omega_p)x_n+B(\Theta_n\omega_p)x_n f_n, \hbox{ for every } n\in \mathbb{Z}.
\end{equation}
As in the deterministic case, see Henry \cite{Henry-1}, this plays an important role
in the proof of robustness of exponential dichotomy. 

\begin{theorem}\label{th-admissibility-for-pertubed-process}
	Let $(\varphi,\Theta)$ be a co-cycle generated by
	$A:\mathbb{Z}\times\Omega\to \mathcal{L}(X)$ 
	and assume that it admits an
	exponential dichotomy with bound $K$ and exponent $\alpha$. 
	Then there exists a $\Theta$-invariant map $\delta$ with
	\begin{equation*}
	0\leq \delta(\omega_p) < 
	\frac{1-e^{-\alpha(\omega_p)}}{1+e^{-\alpha(\omega_p)}},
	\hbox{ for each }\omega_p\in \mathbb{Z}\times\Omega,
	\end{equation*}
	for which, if $B:\mathbb{Z}\times\Omega\to \mathcal{L}(X)$ satisfies
	\begin{equation*}
	\|B(\Theta_k\omega_p)\|_{\mathcal{L}(X)} \leq \delta(\omega_p) K(\omega_p)^{-1}, \forall k\in \mathbb{Z},
	\end{equation*}
	then, for each $\omega_p$ fixed and $\{f_n\}\in l^\infty(\mathbb{Z})$
	the difference equation
	\begin{equation}\label{eq-discrete-perturbed-random-equation}
	x_{n+1}=A(\Theta_n\omega_p)x_n+B(\Theta_n\omega_p)x_n +f_n, \hbox{ for every } n\in \mathbb{Z},
	\end{equation}
	possesses a unique bounded solution $x(\cdot,\omega_p)$.
\end{theorem}

\par The proof of Theorem \ref{th-admissibility-for-pertubed-process} follows step by step the proof presented by Zhou \textit{et al.} \cite{Zhou-Lu-Zhang-1}, which is done for tempered exponential dichotomies.
Some details of the proof are included for they will be needed in forthcoming results.
\begin{proof}
	Let $\omega_p\in \mathbb{Z}\times\tilde{\Omega}$ and
	$f\in l^\infty(\mathbb{Z})$. As a standard procedure, see for instance \cite{Henry-1,Zhou-Lu-Zhang-1},
	we only need to prove that the operator
	\begin{equation*}
	(\Gamma_fx)(n,\omega_p) := \sum_{k=-\infty}^{+\infty} G_{\omega_p}(n,k+1)
	(B(\Theta_k\omega_p) x_k + f_k), \ \ \forall n\in \mathbb{Z}
	\end{equation*}
	has a unique fixed point $x(\cdot,\omega_p)$ in $l^\infty(\mathbb{Z})$.
	\par First, let us prove that $\Gamma_fx(\cdot,\omega_p) \in l^\infty(\mathbb{Z})$, 
	for $x\in l^\infty(\mathbb{Z})$.
	\begin{eqnarray*}
		& &\|(\Gamma_f x)(n,\omega_p)\|_X \leq \sum_{k=-\infty}^{+\infty}
		\|G_{\omega_p}(n,k+1)\|_{\mathcal{L}(X)}
		(\|B(\Theta_k\omega_p)\|_{\mathcal{L}(X)}\, \|x_k\|_X +\|f_k\|_X) \\
		& & \leq \sum_{k=-\infty}^{+\infty} K(\Theta_{k+1}\omega_p) e^{-\alpha(\omega_p)|n-1-k|}
		(\delta(\omega_p) K(\omega)^{-1}\|x_k\|_X+\|f_k\|_X) \\
		& & \leq  \sum_{k=-\infty}^{+\infty}  e^{-\alpha(\omega_p)|n-1-k|} (\delta(\omega_p) \|x\|_{l^\infty}+\|f\|_{l^\infty}K(\omega_p)) \\
		& &\leq \frac{1+e^{-\alpha(\omega_p)}}{1-e^{-\alpha(\omega_p)}} (\delta(\omega_p) \|x\|_{l^\infty}+ \|f\|_{l^\infty}K(\omega_p))
		<+\infty.
	\end{eqnarray*}
	Then, $\Gamma_f(\cdot,\omega_p)(l^\infty(\mathbb{Z}))\subset l^\infty(\mathbb{Z})$.
	Finally, if $x,y\in l^\infty(\mathbb{Z})$, we have that
	\begin{eqnarray*}
		& &\|(\Gamma_f x)(n,\omega_p) - (\Gamma_f y)(n,\omega_p)\|_X  \\
		& &\leq \sum_{k=-\infty}^{+\infty} \|G_{\omega_p}(n,k+1)\|_{\mathcal{L}(X)}\, \|B(\Theta_k\omega_p\|_{\mathcal{L}(X)} \|x_k-y_k\|_X \\
		& &\leq \sum_{k=-\infty}^{+\infty} K(\Theta_{k+1}\omega_p) e^{-\alpha(\omega_p)|n-1-k|}\delta(\omega_p) K(\omega_p)^{-1} \|x_k-y_k\|_X \\
		& &\leq  \frac{1+e^{-\alpha(\omega_p)}}{1-e^{-\alpha(\omega_p)}} \delta(\omega_p) \|x-y\|_{l^\infty}.
	\end{eqnarray*}
	Therefore,
	\begin{equation*}
	\|\Gamma_fx(\cdot,\omega_p)-\Gamma_fy(\cdot,\omega_p)\|_{l^\infty}\leq \frac{1+e^{-\alpha(\omega_p)}}{1-e^{-\alpha(\omega_p)}} \delta(\omega_p) \|x-y\|_{l^\infty},
	\end{equation*}
	thus, we choose $\delta(\omega_p)<\frac{1-e^{-\alpha(\omega_p)}}{1+e^{-\alpha(\omega_p)}}$, 
	thus
	$\Gamma_f(\cdot,\omega_p)$ is a contraction in $l^\infty(\mathbb{Z})$. In this way, we obtain
	$x\in l^\infty(\mathbb{Z})$ such that $x_n(\omega_p)=(\Gamma_fx)(n,\omega_p)$ 
	for each $n\in \mathbb{Z}$, in other words, $x(\cdot,\omega_p)$ is the only solution for 
	\eqref{eq-discrete-perturbed-random-equation}.
\end{proof}

\par The following corollary establishes uniqueness for the family of projections. The proof follows the ideas of \cite[Corollary 7.5]{Carvalho-Langa-Robison-book} and it is included for the readers convenience.
\begin{corollary}\label{cor-uniqueness-projection-discrete}
	If $(\varphi,\Theta)$ admits an exponential dichotomy,
	then the family of projections are uniquely determined.
\end{corollary}
\begin{proof}
	Let $\Pi^{u,(i)}$, for $i=1,2$, 
	be projections associated with an exponential dichotomy of $(\varphi,\Theta)$.
	\par Given $\omega_p\in \mathbb{Z}\times \tilde{\Omega}$ and $z\in X$, define $f_n=0$, for all $n\neq -1$, and $f_{-1}=z$. From Theorem \ref{th-admissibility-for-pertubed-process} with $B=0$, there exists
	$\{x(n,\omega_p): n\in \mathbb{Z}\}$ the unique bounded solution of 
	\begin{equation*}
	x_{n+1}(\omega_p)=A(\Theta_n\omega_p) x_n +f_n, \ \ n\in \mathbb{Z}.
	\end{equation*}
	From the proof of Theorem \ref{th-admissibility-for-pertubed-process} (with $B=0$), it is possible to see that this solution is given by
	\begin{equation}
	x_n(\omega_p)=\sum_{k=-\infty}^{+\infty}G_{\omega_p}^{(i)}(n,k+1)f_k, \  \hbox{ for }
	i=1,2,
	\end{equation}
	where $G^{(i)}$ is the Green function associated with $\Pi^{(i)}$, for $i=1,2$.
	By uniqueness of the solution, we must have that
	$x_0(\omega_p)=\sum_{-\infty}^{+\infty}G_{\omega_p}^{(i)}(0,k+1)f_k=G_{\omega_p}^{(i)}(0,0)f_{-1}=\Pi^{u,(i)}(\omega_p)z$,
	for $i=1,2$.
	Therefore, $\Pi^{u,1}(\omega_p)=\Pi^{u,2}(\omega_p)$ for all
	$\omega_p\in \mathbb{T}\times\tilde{\Omega}$.
\end{proof}

\par For later, we will need a type of Gr\"onwall inequality, see \cite{Barreira-Valls-Nonuniform} for a proof.

\begin{lemma}\label{lemma-grownwall-inequality}
	Let $a$ and $D$ be positive constants and $\gamma,\delta$ nonnegative constants. Suppose that 
	$u:=\{u_n\}_{n\in \mathbb{J}}$ is a nonnegative bounded sequence on 
	$\mathbb{J}=\mathbb{Z}_N^+\ \  (\hbox{or }\mathbb{Z}_N^-)$, such that
	\begin{equation*}
	u_n\leq \gamma D e^{-a(n-N)} + \delta D \sum_{k=N}^{+\infty} e^{-a|n-k-1|}u_k, \ \ n\in \mathbb{J}=\mathbb{Z}_N^+,
	\end{equation*}
	\begin{equation*}
	(\hbox{ or } u_n\leq \gamma D e^{-a(n-N)} + \delta D \sum_{k=-\infty}^{N-1} e^{-a|n-k-1|}u_k, \ \ n\in \mathbb{J}=\mathbb{Z}_N^-,)
	\end{equation*}
	where $\delta < D^{-1}(1-e^{-a})/(1+e^{-a})$.
	\par Then
	\begin{equation*}
	u_n\leq \frac{\gamma D}{1-\delta De^{-a}/(1-e^{-(a+\tilde{a})})} e^{-\tilde{a}(n-N)}, \ \
	n\in \mathbb{J}=\mathbb{Z}_N^+,
	\end{equation*}
	\begin{equation*}
	(\hbox{ or }u_n\leq \frac{\gamma D}{1-\delta De^{-\tilde{b}}/(1-e^{-(a+\tilde{b})})} e^{-\tilde{b}(N-n)}, \ \
	n\in \mathbb{J}=\mathbb{Z}_N^-),
	\end{equation*}
	where $\tilde{a}:=-\ln(\cosh a - [\cosh ^2 a-1-2\delta\sinh a]^{1/2})$ and 
	$\tilde{b}:= \tilde{a}+\ln(1+2\delta D \sinh a)$.
\end{lemma}

\par Now, we state a robustness result of exponential dichotomies for nonautonomous random dynamical systems. 

\begin{theorem}[Robustness of exponential dichotomy]\label{th-discrete-robstness-random-perturbations-of-autonomous-nonuniform-case}
	Let $(\psi,\Theta)$ be a discrete nonautonomous random dynamical system with
	an exponential dichotomy with
	bound $K$ and exponent $\alpha$.
	There exists a $\Theta$-invariant map with
	\begin{equation*}
	0\leq \delta(\omega_p) < 
	\frac{1-e^{-\alpha(\omega_p)}}{1+e^{-\alpha(\omega_p)}},
	\hbox{ for each }\omega_p\in \mathbb{Z}\times\Omega,
	\end{equation*}
	for which, 
	if $(\varphi,\Theta)$ is a discrete nonautonomous 
	random dynamical system
	such that
	\begin{equation}\label{th-discrete-robstness-random-perturbations-of-autonomous-nonuniform-case-hypothesis1}
	\sup_{n\in \mathbb{N}} \{K(\omega_p)\|\psi(1,\Theta_n\omega_p)-\varphi(1,\Theta_n\omega_p)\|_{\mathcal{L}(X)}\}
	\leq \delta(\omega_p),
	\end{equation}
	then $(\varphi,\Theta)$ admits an exponential dichotomy with
	bound
	\begin{equation*}
	M(\omega_p):=K(\omega_p) \left(1+\frac{\delta(\omega_p)}{(1-\rho(\omega_p))(1-e^{-\alpha(\omega_p)})}\right)
	\max\{D_1(\omega_p),D_2(\omega_p)\},
	\end{equation*}
	and exponent
	\begin{equation*}
	\tilde{\alpha}(\omega_p):=-\ln(\cosh \alpha (\omega_p)- [\cosh ^2 \alpha(\omega_p)-1-2\delta(\omega_p)\sinh \alpha(\omega_p)]^{1/2}),
	\end{equation*}
	where 
	$\rho(\omega_p):=\delta(\omega_p)(1+e^{-\alpha(\omega_p)})/(1-e^{-\alpha(\omega_p)})$,
	$D_1(\omega_p):=[1-\delta(\omega_p) e^{-\alpha(\omega_p)}/(1-e^{-\alpha(\omega_p)-\tilde{\alpha}(\omega_p)})]^{-1}$,
	$D_2:=[1-\delta(\omega_p) e^{-\tilde{\beta}(\omega_p)}/(1-e^{-\alpha(\omega_p)-\tilde{\beta}(\omega_p)})]^{-1}$
	and
	$\tilde{\beta}(\omega_p):=\tilde{\alpha}(\omega_p)+\ln(1+2\delta(\omega_p)\sinh\alpha(\omega_p))$.
\end{theorem}

\begin{remark}
	The proof of Theorem \ref{th-discrete-robstness-random-perturbations-of-autonomous-nonuniform-case} 
	is entirely based on the proof of Zhou \textit{et al.} in \cite{Zhou-Lu-Zhang-1}. 
	We remark that their proof works for any co-cycle defined of a noncompact symbol space. 
	In fact, let 
	$\varphi$ be a linear co-cycle driving by a flow $\Sigma\times \Omega \ni (\sigma,\omega)\mapsto\Theta_t(\sigma,\omega)=(\theta^1_t\sigma,\theta^2_t\omega)$, where 
	$\theta^1_t:\Sigma\to \Sigma$ is a flow in a metric space $\Sigma$ and
	$\theta^2_t:\Omega\to \Omega$ is a random flow, and $t\in \mathbb{Z}$. 
	Then, following the ideas of \cite{Zhou-Lu-Zhang-1} and this present work, it is possible to provide a suitable definition of tempered exponential dichotomy for a general linear co-cycle $(\varphi,\Theta)_{(X,\Sigma\times \Omega)}$ and to prove a robustness result for it. 
	\par In this work we choose to work with the case where the bound $K$ is $\Theta$-invariant, because we want to understand the effect of a bounded noise on the hyperbolicity of an autonomous problem, and therefore it is not expect to obtain tempered exponential dichotomies which means that, the mapping $t\mapsto K(\theta_t\omega_p)$ has a sub-exponential growth, see \cite{Zhou-Lu-Zhang-1}.
\end{remark}
\par The proof of Theorem \ref{th-discrete-robstness-random-perturbations-of-autonomous-nonuniform-case} follows the same line of arguments of \cite[Theorem 13]{Zhou-Lu-Zhang-1}. One of the main differences is that we use the evolution processes $\Phi_{\omega_p}$ associated with a given co-cycle $\varphi$, see Remark \ref{remark-associated-processes}. This simple fact provides fundamental ideas for the proof on the continuous case (Subsection \ref{subsec-ed-continuous}), and makes the writing much simpler for the discrete case. Some of the details of the proof are included for the readers convenience.
\par To prove Theorem 
\ref{th-discrete-robstness-random-perturbations-of-autonomous-nonuniform-case},
we first prove a lemma that provides a decomposition of the space $X$
associated with the linear nonautonomous random dynamical system generated by the perturbed homogeneous problem
\eqref{eq-random-perturbed-nonhomogen-linear}.

\begin{lemma}\label{lemma-key-to-prove-robstness-of-discrete}
	Assume conditions of Theorem 
	\ref{th-admissibility-for-pertubed-process}
	are satisfied. Then, for each 
	$\omega_p\in \mathbb{Z}\times\tilde{\Omega}$, 
	$X$ admits 
	a decomposition
	\begin{equation*}
	X=V^+(\omega_p)\bigoplus V^-(\omega_p).
	\end{equation*}
	Furthermore, $\psi(n,\omega_p) V^+(\omega_p)\subset V^+(\Theta_n\omega_p)$, 
	$\psi(n,\omega_p) V^-(\omega_p)=V^-(\Theta_n\omega_p)$ and\\
	$\psi(n,\omega_p)|_{V^-(\omega_p)}$ is an isomorphism for each $n\geq 0$,
	where $(\psi,\Theta)$ is the co-cycle associated with $A+B$ in problem
	\eqref{eq-random-perturbed-nonhomogen-linear}.
\end{lemma}

\begin{proof}
	\par Thanks to Theorem \ref{th-admissibility-for-pertubed-process}, 
	there exists a $\Theta$-invariant measurable map $\delta$ such that
	for each perturbation satisfying condition
	\eqref{th-admissibility-for-pertubed-process}, the pair
	$(l^\infty(\mathbb{Z}),l^\infty(\mathbb{Z}))$ is admissible for
	\eqref{eq-discrete-perturbed-random-equation}, for each $\omega_p\in \mathbb{Z}\times \tilde{\Omega}$. 
	\par Hence, take $(\psi,\Theta)$ a co-cycle satisfying
	\eqref{th-discrete-robstness-random-perturbations-of-autonomous-nonuniform-case-hypothesis1}
	and let $B(\omega_p):=\psi(1,\omega_p)-\varphi(1,\omega_p)$ be a linear bounded perturbation
	and consider problem
	\eqref{eq-random-perturbed-nonhomogen-linear}. Also, for each 
	$\omega_p$ we define the evolution process
	$\Psi_{\omega_p}:=\{\psi_{n,m}(\omega_p): n\geq m\}$ associated with $(\psi,\Theta)$.
	\par Define the following sets
	\begin{eqnarray*}
		V^+(\omega_p)&:=&\{z\in X \, ; \, 
		\sup_{n\in \mathbb{N}}\|\psi(n,\omega_p)z\|_{\mathcal{L}(X)} <+\infty \},\\
		V^-(\omega_p)&:=&\{z\in X \, ; \, \hbox{ there is a backwards bounded solution for } 
		\Psi_{\omega_p} \hbox{ through }z \}.
	\end{eqnarray*}
	These are the candidates to be the subspaces that provide the desire decomposition. 
	We prove this fact in four steps.
	\par \textbf{Step 1:} $V^+(\omega_p)$ and $V^-(\omega_p)$ are closed subspaces, 
	for each $\omega_p\in \mathbb{Z}\times \tilde{\Omega}$.
	\par First, note that as $\psi(n,\omega_p)$ is a bounded linear operator it follows that for each 
	$\omega_p$ fixed we have that
	$V^+(\omega_p)$ is a subspace of $X$. 
	We prove now that it is closed. 
	Let $z\in V^+(\omega_p)$, then
	$x_n(\omega_p):=\psi(n,\omega_p)z$ is a bounded solution of
	\eqref{eq-random-perturbed-nonhomogen-linear} for $n\in \mathbb{N}$.
	Then $x(\cdot,\omega_p)$ satisfies
	\begin{eqnarray*}
		x_n(\omega_p)&=&\varphi(n,\omega_p)\Pi^s(\omega_p)z+
		\sum_{k=0}^{n-1}\varphi(n-k-1,\Theta_{k+1}\omega_p)\Pi^s(\theta_{k+1}\omega_p)
		B(\Theta_{k}\omega_p)x_k(\omega_p)\\
		&+&\sum_{k=n}^{+\infty}\varphi(n-k-1,\Theta_{k+1}\omega_p)\Pi^u(\theta_{k+1}\omega_p)B(\Theta_{k}\omega_p)x_k(\omega_p), \ \ n\in \mathbb{N}.
	\end{eqnarray*}
	Hence
	\begin{equation*}
	\|x_n(\omega_p)\|_X\leq K(\omega_p)e^{-\alpha(\omega_p)n}\|z\|_X+
	\delta(\omega_p)\sum_{k=0}^{+\infty}e^{-\alpha(\omega_p)|n-1-k|}
	\|x(k,\omega_p)\|_X, \ \ n\in \mathbb{N}.
	\end{equation*}
	Since 
	$\delta(\omega_p)\leq (1-e^{-\alpha(\omega_p)})/(1+e^{-\alpha(\omega_p)})$,
	by Lemma \ref{lemma-grownwall-inequality},
	we obtain
	\begin{equation*}
	\|x_n(\omega_p)\|_X\leq
	\frac{K(\omega_p)\|z\|_X}{1-\delta(\omega_p)e^{-\alpha(\omega_p)}/
		(1-e^{\alpha(\omega_p)+\tilde{\alpha}(\omega_p)})} e^{-\alpha(\omega_p)n},
	\ \ n\in \mathbb{N}.
	\end{equation*}
	Let $\{z_j\}_{j\in\mathbb{N}}$ be a sequence in
	$V^+(\omega_p)$ such that $z_j\to z$ as 
	$j\to +\infty$.
	Note that
	\begin{equation*}
	\|\psi(n,\omega_p)(z_j-z_i)\|_X\leq
	K(\omega_p)D_1(\omega_p) e^{-\alpha(\omega_p)n}\|z_j-z_i\|_X,
	\ \ n\in \mathbb{N},
	\end{equation*}
	which implies that
	$\{\psi(n,\omega_p)(z_j)\}_{j\in\mathbb{N}}$ is a Cauchy sequence in $X$
	for each fixed $\omega_p$ and $n\in \mathbb{N}$.
	Then
	\begin{equation*}
	\|\psi(n,\omega)z\|_{X}\leq D_1(\omega_p)K(\omega_p)e^{-\alpha(\omega_p)n}
	\|z\|_X, \hbox{ for } n\in \mathbb{N},
	\end{equation*}
	and therefore $V^+(\omega_p)$ is closed for each
	$\omega_p\in \mathbb{Z}\times \tilde{\Omega}$.
	\par Now we show that  $V^-(\omega_p)$ is a closed subspace. 
	Note that, since $\psi(n,\omega_p)$ is a linear operator for all
	$n\in \mathbb{N}$, it follows that $V^-(\omega_p)$ is subspace of $X$.
	Let us prove that it is closed; it will be very similar to the proof for $V^+(\omega_p)$ and, for this reason, some steps will be omitted. 
	We claim that each $z\in V^-(\omega_p)$ satisfies
	\begin{equation}\label{eq-prop-claim2-for-V^-}
	\|\xi(n)\|_{X}\leq D_2(\omega_p)K(\omega_p)e^{\tilde{\beta}(\omega_p)n}
	\|z\|_X, \ \ n\leq 0,
	\end{equation}
	where $\xi$ is a backwards bounded solution of 
	$\Psi_{\omega_p}:=\{\psi(t-s,\Theta_s\omega_p)\}$ through $z$, $D_2,$ and $\tilde{\beta}$
	are in Theorem \ref{th-discrete-robstness-random-perturbations-of-autonomous-nonuniform-case}.
	Indeed, note that
	\begin{eqnarray*}
		\xi(n)&=&\varphi(n,\omega_p)\Pi^u(\omega_p)z-
		\sum_{k=n}^{-1}\varphi(n-k-1,\Theta_{k+1}\omega_p)\Pi^u(\theta_{k+1}\omega_p)
		B(\Theta_{k}\omega_p)\xi(k)\\
		&+&\sum_{k=-\infty}^{n-1}\varphi(n-k-1,\Theta_{k+1}\omega_p)\Pi^s(\theta_{k+1}\omega_p)
		B(\Theta_{k}\omega_p)\xi(k).
	\end{eqnarray*}
	Thus
	\begin{equation*}
	\|\xi(n)\|_X\leq K(\omega_p)e^{\alpha(\omega_p)n}\|z\|_X+
	\delta(\omega_p)\sum_{k=-\infty}^{-1}e^{-\alpha(\omega_p)|n-1-k|}
	\|\xi(k)\|_X, \ \ n\leq 0,
	\end{equation*}
	and Lemma \ref{lemma-grownwall-inequality} implies \eqref{eq-prop-claim2-for-V^-}.
	Then, for a sequence 
	$\{z_j\}_{j\in \mathbb{Z}}$ in $V^-(\omega_p)$ such that
	$z_j\to z$ as $j\to +\infty$, we will show that $z\in V^-(\omega_p)$.
	In fact, for each $j\in \mathbb{N}$ there exists 
	a backwards bounded solution $\xi_j$ of
	$\Psi_{\omega_p}$ such that
	$\xi_j(0)=z_j$. 
	Thanks to \eqref{eq-prop-claim2-for-V^-}
	it is possible to construct a backwards solution $\xi$
	through $z$, such that
	$\xi(n)=\lim_{j\to +\infty}\xi_j(n)$
	for each fixed $n\leq 0$.
	\par To see that $z\in V^-(\omega_p)$ we have to show that
	$\xi$ is bounded. In fact, for all fixed $n\leq 0$ we have
	\begin{eqnarray*}
		\|\xi(n)\|_X&\leq& \|\xi(n)-\xi_j(n)\|_X+\|\xi_j(n)\|_X\\
		&\leq& \|\xi(n)-\xi_j(n)\|_X+K(\omega_p)D_2(\omega_p)e^{\tilde{\beta}(\omega_p)n}\|z_j\|_X,
	\end{eqnarray*}
	which yields to
	\begin{equation*}
	\|\xi(n)\|_X\leq K(\omega_p)D_2(\omega_p)e^{\tilde{\beta}(\omega_p)n}\|z\|_X, \ \  n\leq 0,
	\end{equation*}
	and concludes the proof of the first step.
	\par \textbf{Step 2:} $V^+(\omega_p)$ is positively invariant and 
	$V^-(\omega_p)$ is invariant.
	\par Note that
	if $z\in V^+(\omega_p)$, 
	then
	$\{\psi(k,\omega_p)z\}_{k\in\mathbb{N}}$ 
	is a bounded sequence,
	then for a fixed $m\geq0$, by the co-cycle property
	\begin{equation*}
	\psi(k,\Theta_{m}\omega_p)\psi(m,\omega_p)z=
	\psi(m+k,\omega_p)z
	\end{equation*}
	is an element of the bounded sequence $\{\psi(k,\omega_p)z\}_{k\in\mathbb{N}}$ for each $k>0$,
	and therefore is also bounded. 
	Thus 
	\begin{equation*}
	\psi(m,\omega_p)z\in V^+(\Theta_m \omega_p), \hbox{ and }
	\psi(m,\omega_p)V^+(\omega_p)\subset V^+(\Theta_{m}\omega_p).
	\end{equation*}
	\par Let us prove now that 
	$\psi(k,\omega_p) V^-(\omega_p)\subset V^-(\Theta_k\omega_p)$, for all $k\geq 0$.
	Let $z\in V^-(\omega_p)$ and $\xi$ a backwards bounded solution of
	$\Psi_{\omega_p}$ such that $\xi(0)=z$.
	\par For fixed $k\geq 0$ define $\tilde{z}=\psi(k,\omega_p)z$ and
	\begin{equation*}
	\tilde{\xi}(n)= \left\{ 
	\begin{array}{l l} 
	\psi(n,\omega_p)z, 
	&  \quad \hbox{if } n\in\{0,1,\cdots,k\}
	\\ \xi(n),  
	& \quad \hbox{if } n\leq 0. 
	\end{array} 
	\right.
	\end{equation*} 
	Notice that $\tilde{\xi}$ is a backwards bounded solution for
	$\Psi_{\Theta_{k}\omega_p}$ 
	and $\tilde{\xi}(k)=\tilde{z}$, which implies that
	$\tilde{z}\in V^-(\Theta_k\omega_p)$.
	\par Now, let $z\in V^-(\Theta_k\omega_p)$, 
	then there is $\hat{\xi}$ a backwards bounded solution for 
	$\Psi_{\theta_k\omega_p}$ 
	such that $\hat{\xi}(0)=z$. 
	In particular
	\begin{equation*}
	z=\psi_{0,-k}(\Theta_k\omega_p)\hat{\xi}(-k)=
	\psi(k,\omega_p)\hat{\xi}(-k),
	\end{equation*}
	thus it is enough to prove that $\hat{\xi}(-k)\in V^-(\omega_p)$.
	Define $\xi(n):=\hat{\xi}(n-k)$ for $n\leq 0$. 
	It is clear that $\xi$ is a backwards bounded solution for 
	$\Psi_{\omega_p}$ and through $\hat{\xi}(-k)$. By definition
	we conclude that $\hat{\xi}(-k)\in V^-(\omega_p)$ and Step 2 is complete.
	\par \textbf{Step 3:} $X=V^+(\omega_p)\oplus V^-(\omega_p)$.
	\par Let $z\in V^+(\omega_p)\cap V^-(\omega_p)$. Then by the definition of these subspaces 
	we see that there is a bounded global solution 
	$\xi$ for $\Psi_{\omega_p}$
	through $z$. 
	On the other hand, applying Theorem \ref{th-admissibility-for-pertubed-process} with 
	$f=0$, we know that the only complete bounded solution is the $\xi=0$, thus by uniqueness we have that $\xi(n)=0$, for all $n\in\mathbb{Z}$, in particular
	$z=0$. Then $V^+(\omega_p)\cap V^-(\omega_p)=\{0\}$.
	\par Let $z\in X$ and define $f_n=0$ for all $n\neq -1$, and $f_{-1}=z/K(\omega_p)$.
	Thus by Theorem \ref{th-admissibility-for-pertubed-process} there exists 
	$x(\cdot,\omega_p)=\{x_n(\omega_p)\}_{n\in \mathbb{Z}}\in l^\infty(\mathbb{Z})$ the solution of \eqref{eq-random-perturbed-nonhomogen-linear}. 
	This solution satisfies for $n\geq m$
	\begin{equation*}
	x_n(\omega_p)=\psi_{n,m}(\omega_p)x_m(\omega_p
	)+\sum_{k=m}^{n-1}\psi(n-k-1,\Theta_{k+1}\omega_p)f_k.
	\end{equation*}
	Rewriting this we obtain
	\begin{eqnarray*}
		x_n(\omega_p)&=&\psi(n,\omega_p)x_{0}(\omega_p), \ \  n\geq 0, \\
		x_{0}(\omega_p)&=&\psi(1,\Theta_{-1}\omega_p)x_{-1}(\omega_p) +z/K(\omega_p), \\
		x_n(\omega_p)&=&\psi_{n,m}(\omega_p) x_m, \ \ m\leq n\leq -1.
	\end{eqnarray*}
	Thus, since $x$ is bounded, we see that $x_{0}(\omega_p)\in V^+(\omega_p)$.
	\par Note that $\psi(1,\Theta_{-1}\omega_p)x_{-1}(\omega_p)\in V^-(\omega_p)$.
	In fact, define
	$\xi(0):=x_0(\omega_p)-z/K(\omega_p)$ and $\xi(n):=x_n(\omega_p)$ for $n\leq-1$. 
	Then $\xi$ is a backwards bounded solution through 
	$\psi(1,\Theta_{-1}\omega_p)x_{-1}(\omega_p)=x_{0}(\omega_p)-z/K(\omega_p)$,
	which means that
	$x_{0}(\omega_p)-z/K(\omega_p)\in V^-(\omega_p)$.
	\par Therefore, 
	\begin{equation*}
	z=x_{0}(\omega_p)K(\omega_p) - (x_{0}(\omega_p)K(\omega_p)-z)\in V^+(\omega_p)+ V^-(\omega_p),
	\end{equation*} 
	which completes the proof of Step 3.
	\par \textbf{Step 4: } $\psi(m,\omega_p)|_{V^-(\omega_p)}: V^-(\omega_p)\rightarrow V^-(\Theta_m\omega_p)$ is an isomorphism, 
	for $m\geq 0$.
	\par By Step 2, we already have that $\psi(m,\omega_p)|_{V^-(\omega_p)}$ is surjective,
	so now we
	show that is injective. 
	Given $z\in V^-(\omega_p)$, then from 
	the proof of Step 2, we know that there exists a backwards bounded solution 
	$\tilde{\xi}$ of
	$\Psi_{\Theta_m\omega_p}$ such that
	$\tilde{\xi}(0)=\psi(m,\omega_p)z$ and 
	$\tilde{\xi}(-m)=z$.
	Thus, from \eqref{eq-prop-claim2-for-V^-}
	\begin{equation*}
	\|\tilde{\xi}(n)\|_X\leq D_2(\omega_p)K(\Theta_m\omega_p)e^{\tilde{\beta}(\omega_p)m}\|\psi(m,\omega_p)z\|_X,
	\ \ n\leq 0.
	\end{equation*}
	In particular, 
	\begin{equation*}
	\|z\|_X\leq D_2(\omega_p)K(\Theta_m\omega_p)e^{-\tilde{\beta}(\omega_p)m}\|\psi(m,\omega_p)z\|_X.
	\end{equation*}
	Hence, $\psi(m,\omega_p)|_{V^-(\omega_p)}$ is injective, and the proof of the lemma
	is complete.
\end{proof}

\par Now, we are ready to prove our result on the robustness of exponential dichotomies for discrete nonautonomous random dynamical systems:
\medskip
\par \noindent \textbf{Proof of Theorem
	\ref{th-discrete-robstness-random-perturbations-of-autonomous-nonuniform-case}.}
From Lemma \ref{lemma-key-to-prove-robstness-of-discrete}, for each 
$\omega_p\in \mathbb{Z}\times \tilde{\Omega}$, we have
$X=V^+(\omega_p)\oplus V^-(\omega_p)$, so define
$\widetilde{\Pi}^s(\omega_p)$ as the projection from $X$ onto 
$V^+(\omega_p)$, and
$\widetilde{\Pi}^u(\omega_p):=Id_X-\widetilde{\Pi}^s(\omega_p)$.
By the invariance of $V^-(\omega_p)$ and positively invariance of $V^+(\omega_p)$ we have that
\begin{equation*}
\psi(n,\omega_p)\widetilde{\Pi}^s(\omega_p)=
\widetilde{\Pi}^s(\Theta_n\omega_p)\psi(n,\omega_p), \hbox{ for all }n\in \mathbb{N}.
\end{equation*}
Again, by Lemma \ref{lemma-key-to-prove-robstness-of-discrete},
$\psi(n,\omega_p)|_{R(\widetilde{\Pi}^u(\omega_p))}$ is an isomorphism and we define\\
$\psi(-n,\Theta_n\omega_p)$ as the inverse of the map 
$\psi(n,\omega_p)|_{V^-(\omega_p)}$.
\par Let $z\in X$ and define $f_n=0$ for all $n\neq -1$, and $f_{-1}=z/K(\omega_p)$.
Then Equation \eqref{eq-discrete-perturbed-random-equation} has a unique bounded solution
$x(\cdot,\omega_p)$, with $x_n(\omega_p)\in V^+(\Theta_n\omega_p)$ for all $n\in\mathbb{N}$,
$x_n(\omega_p)\in V^-(\Theta_n\omega_p)$ for all $n\leq -1$, and
$x_0(\omega_p)-z/K(\omega_p)\in V^-(\omega_p)$. 
Note that
\begin{equation}\label{eq-th-robstness-positive}
x_n(\omega_p)=\frac{1}{K(\omega_p)}\psi(n,\omega_p)\widetilde{\Pi}^s(\omega_p)z, \hbox{ for all }n\in \mathbb{N},
\end{equation}
and
\begin{equation}\label{eq-th-robstness-negative}
x_m(\omega_p)=\frac{-1}{K(\omega_p)}\psi(m,\omega_p)\widetilde{\Pi}^u(\omega_p)z, \hbox{ for all }m\leq -1.
\end{equation}
\par Now, since $x(\cdot,\omega_p)$ is a solution of 
$x_{n+1}=A(\Theta_n\omega_p)x_n+B(\Theta_n\omega_p)x_n+f_n$, for all 
$n\in \mathbb{Z}$ we have that
\begin{equation*}
x_n(\omega_p)=G_{\omega_p}(n,0)f_{-1}+\sum_{-\infty}^{+\infty}G_{\omega_p}(n,k+1) B(\Theta_k\omega_p)x_k(\omega_p),
\end{equation*}
where $G$ is the Green function associated with the co-cycle $(\varphi,\Theta)$ and family of 
projections $\Pi^s$. Hence,
\begin{eqnarray*}
	\|x_n(\omega_p)\|_X&\leq &e^{-\alpha(\omega_p)|n|}\|z\|_X +\sum_{-\infty}^{+\infty}
	K(\omega_p)e^{-\alpha(\omega_p)|n-k-1|} 
	\|B(\Theta_k\omega_p)\|_{\mathcal{L}(X)} \|x_k(\omega_p)\|_X\\
	&\leq &e^{-\alpha(\omega_p)|n|}\|z\|_X+ \|x(\cdot,\omega_p)\|_{l^\infty(\mathbb{Z})}\delta(\omega_p)
	\frac{1+e^{-\alpha(\omega_p)}}{1-e^{-\alpha(\omega_p)}}.
\end{eqnarray*}
Thus,
\begin{equation*}
\|x(\cdot,\omega_p)\|_{l^\infty(\mathbb{Z})}\leq \frac{\|z\|_X}{1-\rho(\omega_p)},
\end{equation*}
where $\rho(\omega_p)=\delta(\omega_p)(1+e^{-\alpha(\omega_p)})/(1-e^{-\alpha(\omega_p)})$.
In particular, from \eqref{eq-th-robstness-positive}
\begin{equation}\label{eq-th-robstness-positive-uniform}
\|\psi(n,\omega_p)\widetilde{\Pi}^s(\omega_p)z\|_X\leq 
\frac{K(\omega_p)}{1-\rho(\omega_p)}\|z\|_X, \ n\geq 0,
\end{equation}
and from \eqref{eq-th-robstness-negative}
\begin{equation}\label{eq-th-robstness-negative-uniform}
\|\psi(m,\omega_p)\widetilde{\Pi}^u(\omega_p)z\|_X\leq 
\frac{K(\omega_p)}{1-\rho(\omega_p)}\|z\|_X, \ m\leq -1.
\end{equation}
\par Now, we use \eqref{eq-th-robstness-negative-uniform} to obtain a better estimate than \eqref{eq-th-robstness-positive-uniform} for the norm of
$\psi(n,\omega_p)\widetilde{\Pi}^s(\omega_p)$.
Since
$\widetilde{\Pi}^s(\omega_p)z\in V^+(\omega_p)$, we know that
$\psi(n,\omega_p)\widetilde{\Pi}^s(\omega_p)z$ defines a bounded solution
for 
\begin{equation*}
x_{n+1}(\omega_p)=A(\Theta_n\omega_p)x_n(\omega_p)+B(\Theta_n\omega_p)x_n(\omega_p), \ \  n\geq 0,
\end{equation*}
and since 
$\widetilde{\Pi}^u(\omega_p)z\in V^-(\omega_p)$, there exists
$\xi:(-\infty,0]\cap\mathbb{Z}\rightarrow X$ backwards bounded solution of
$\Psi_{\omega_p}$ such that
$\xi(0)=\widetilde{\Pi}^u(\omega_p)z$. 
In particular,
$\xi(n):=\psi(n,\omega_p)\widetilde{\Pi}^u(\omega_p)z$, for $n\leq -1.$
Hence
\begin{eqnarray*}
\psi(n,\omega_p)\widetilde{\Pi}^s(\omega_p)z&=&
\varphi(n,\omega_p)\Pi^s(\omega_p)\widetilde{\Pi}^s(\omega_p)z\\
&+&\sum_{k=0}^{+\infty}
G_{\omega_p}(n,k+1)B(\Theta_k\omega_p)\psi(k,\omega_p)\widetilde{\Pi}^s(\omega_p)z,  \ n\geq 0,
\end{eqnarray*}
and
\begin{eqnarray*}
	\psi(n,\omega_p)\widetilde{\Pi}^u(\omega_p)z&=&
	\varphi(n,\omega_p)\Pi^u(\omega_p)\widetilde{\Pi}^u(\omega_p)z\\
	&+&\sum_{k=-\infty}^{-1} G_{\omega_p}(n,k+1)B(\Theta_k\omega_p)\psi(k,\omega_p)\widetilde{\Pi}^u(\omega_p)z, \  n\leq 0.
\end{eqnarray*}
For $n=0$ we obtain
\begin{eqnarray}\label{eq-3.39}
& &\widetilde{\Pi}^u(\omega_p)z=
\Pi^u(\omega_p)\widetilde{\Pi}^u(\omega_p)z\\
& &+\sum_{k=-\infty}^{-1} \varphi(-k-1,\Theta_{k+1}\omega_p)\Pi^s(\Theta_{k+1}\omega_p)B(\Theta_k\omega_p)\psi(k,\omega_p)\widetilde{\Pi}^u(\omega_p)z.
\end{eqnarray}
Thus
\begin{eqnarray*}
	\psi(n,\omega_p)\widetilde{\Pi}^s(\omega_p)z&=&
	\varphi(n,\omega_p)\Pi^s(\omega_p)z-\varphi(n,\omega_p)\Pi^s(\omega_p)\widetilde{\Pi}^u(\omega_p)z\\
	&+&\sum_{k=0}^{+\infty}
	G_{\omega_p}(n,k+1)B(\Theta_k\omega_p)\psi(k,\omega_p)\widetilde{\Pi}^s(\omega_p)z,  \ n\geq 0,
\end{eqnarray*}
hence, from \eqref{eq-3.39}
\begin{eqnarray*}
	& &\psi(n,\omega_p)\widetilde{\Pi}^s(\omega_p)z=
	\varphi(n,\omega_p)\Pi^s(\omega_p)z\\
	& &-\sum_{k=-\infty}^{-1} \varphi(n-k-1,\Theta_{k+1}\omega_p)\Pi^s(\Theta_{k+1}\omega_p)B(\Theta_k\omega_p)\psi(k,\omega_p)\widetilde{\Pi}^u(\omega_p)z\\
	& &+\sum_{k=0}^{+\infty}
	G_{\omega_p}(n,k+1)B(\Theta_k\omega_p)\psi(k,\omega_p)\widetilde{\Pi}^s(\omega_p)z,  \ n\geq 0.
\end{eqnarray*}
Thus, from \eqref{eq-th-robstness-negative-uniform}
\begin{eqnarray*}
	\|\psi(n,\omega_p)\widetilde{\Pi}^s(\omega_p)\|_{\mathcal{L}(X)}&\leq&
	K(\omega_p)\Bigg( 1+\frac{\delta(\omega_p)}{(1-\rho(\omega_p))(1-e^{\alpha(\omega_p)})}\Bigg)e^{-\alpha(\omega_p)n}\\
	&+&\delta(\omega_p)\sum_{k=0}^{+\infty}e^{-\alpha(\omega_p)|n-k-1|}
	\|\psi(k,\omega_p)\widetilde{\Pi}^s(\omega_p)\|_{\mathcal{L}(X)},  \ n\geq 0.
\end{eqnarray*}
Again, from \eqref{eq-th-robstness-negative-uniform} we have that 
\begin{equation*}
\|\psi(n,\omega_p)\widetilde{\Pi}^s(\omega_p)\|_{\mathcal{L}(X)}\leq K_1(\omega_p):=1+K(\omega_p)/(1-\rho(\omega_p)),  \ \forall n\geq 0.
\end{equation*}
Hence, 	$u_n:=K_1(\omega_p)^{-1}\|\psi(n,\omega_p)\widetilde{\Pi}^s(\omega_p)\|_{\mathcal{L}(X)}$ is uniformly bounded
for each $n\geq 0$ and $\omega_p$. 
Thus, 
\begin{equation*}
u_n\leq \frac{K(\omega_p)}{K_1(\omega_p)}
\Bigg[1+\frac{\delta(\omega_p)}{(1-\rho(\omega_p))(1-e^{-\alpha(\omega_p)})}\Bigg]
e^{-\alpha(\omega_p)n}+
\delta(\omega_p)\sum_{k=0}^{+\infty}e^{-\alpha(\omega_p)|n-k-1|}u_k.
\end{equation*}
Then, by Lemma \ref{lemma-grownwall-inequality} 
\begin{equation*}
\|\psi(n,\omega_p)\widetilde{\Pi}^s(\omega_p)\|_{\mathcal{L}(X)}\leq 
\frac{K(\omega_p)\bigg[1+\frac{\delta(\omega_p)}{(1-\rho(\omega_p))(1-e^{-\alpha(\omega_p)})}\bigg]}{
	1-\delta(\omega_p)e^{-\alpha(\omega_p)}/(1-e^{-\alpha(\omega_p)-\tilde{\alpha}(\omega_p)})} e^{-\tilde{\alpha}(\omega_p)n}, \  n\geq 0.
\end{equation*}
Similarly, 
we use \eqref{eq-th-robstness-positive-uniform} and 
\begin{eqnarray*}
	& &\psi(n,\omega_p)\widetilde{\Pi}^u(\omega_p)z=
	\varphi(n,\omega_p)\Pi^u(\omega_p)z\\
	& &+\sum_{k=0}^{+\infty} \varphi(n-k-1,\Theta_{k+1}\omega_p)\Pi^u(\Theta_{k+1}\omega_p)B(\Theta_k\omega_p)\psi(k,\omega_p)\widetilde{\Pi}^s(\omega_p)z\\
	& &+\sum_{k=-\infty}^{-1}
	G_{\omega_p}(n,k+1)B(\Theta_k\omega_p)\psi(k,\omega_p)\widetilde{\Pi}^u(\omega_p)z, \  n\leq 0
\end{eqnarray*}
to obtain
\begin{eqnarray*}
	& &\|\psi(n,\omega_p)\widetilde{\Pi}^u(\omega_p)\|_{\mathcal{L}(X)}\leq
	K(\omega_p)\Bigg( 1+\frac{\delta(\omega_p)e^{-\alpha(\omega_p)}}{(1-\rho(\omega_p))(1-e^{\alpha(\omega_p)})}\Bigg)e^{\alpha(\omega_p)n}\\
	& &+\,\delta(\omega_p)\sum_{k=-\infty}^{-1}e^{-\alpha(\omega_p)|n-k-1|}
	\|\psi(k,\omega_p)\widetilde{\Pi}^u(\omega_p)\|_{\mathcal{L}(X)},  \ n\leq 0;
\end{eqnarray*}
and from Lemma \ref{lemma-grownwall-inequality} 
\begin{equation*}
\|\psi(n,\omega_p)\widetilde{\Pi}^u(\omega_p)\|_{\mathcal{L}(X)}\leq 
\frac{K(\omega_p)\bigg[1+\frac{\delta(\omega_p)e^{-\alpha(\omega_p)}}{(1-\rho(\omega_p))(1-e^{-\alpha(\omega_p)})}\bigg]}{
	1-\delta(\omega_p)e^{-\tilde{\beta}(\omega_p)}/(1-e^{-\alpha(\omega_p)-\tilde{\beta}(\omega_p)})} e^{\tilde{\beta}(\omega_p)n}, \  n\leq 0.
\end{equation*}
\par Finally, we prove that, for each $p\in \mathbb{Z}$,
the operator 
$\Pi^s(p,\cdot):\Omega\to \mathcal{L}(X)$ is strongly measurable.
Note that
\begin{equation}\label{eq-3.51}
\widetilde{G}_{\omega_p}(n,0)z=G_{\omega_p}(n,0)z+
\sum_{k=-\infty}^{+\infty}G_{\omega_p}(n,k+1)B(\Theta_{k}\omega_p)\tilde{G}_{\omega_p}(k,0)z,
\forall n\in \mathbb{Z}, \ z\in X,
\end{equation}
where $\widetilde{G}$ is the Green function of $(\psi,\Theta)$ and family of projection $\widetilde{\Pi}^s$.
Let $T(\omega,p,z)$ be the operator defined on 
$l^{\infty}(\mathbb{Z})$ by
\begin{equation*}
T(\omega,p,z)g_n:=G_{\omega_p}(n,0)z+\sum_{k=-\infty}^{+\infty}
G_{\omega_p}(n,k+1)B(\Theta_{k}\omega_p)g_k.
\end{equation*}
Note that
$T(\omega,p,z)$ has a unique fixed point $g\in l^\infty(\mathbb{Z})$, and
from \eqref{eq-3.51} we know that 
$g=\{\widetilde{G}_{\omega_p}(n,0)z\}_{n\in \mathbb{Z}}$.
Moreover, since $T(\cdot,p,z)$ is $(\mathcal{F}, \mathcal{B}(X))$-measurable
and  
$\{\widetilde{G}_{\omega_p}(n,0)z\}_{n\in \mathbb{Z}}$ is a limit of the iteration of $T(\cdot,p,z)$ starting at $0\in l^\infty(\mathbb{Z})$ we conclude that
$\{\widetilde{G}_{(p,\cdot)}(n,0)z\}_{n\in \mathbb{Z}}$ is a $(\mathcal{F}, \mathcal{B}(X))$-measurable sequence, which implies that 
$\Pi^s(p,\cdot)$ is strongly measurable.
\par Therefore $(\psi,\Theta)$ admits an exponential dichotomy on a $\Theta$-invariant subset $\tilde{\Omega}$ of full measure with projection
$\widetilde{\Pi}^s$, exponent $\tilde{\alpha}$, and bound $\tilde{K}$, the proof is complete.
\qed

\begin{remark}
	Note that, thanks to the approach given by \cite{Zhou-Lu-Zhang-1} we are able to provide
	explicit expression for the bound $\tilde{K}$ and exponent $\tilde{\alpha}$ for the obtained exponential dichotomy. This also happens in the continuous case.
\end{remark}

\par Now, we prove a continuous dependence result of the projections associated with exponential dichotomy for co-cycle. The proof is inspired in \cite[Theorem 7.9]{Carvalho-Langa-Robison-book}.

\begin{theorem}\label{th-continuity-projection-discrete}
	Suppose that  $\varphi$ and $(\psi,\Theta)$ are nonautonomous random dynamical systems and that they
	admit an exponential dichotomy with projections 
	$\Pi_\varphi^s$ and $\Pi_\psi^s$,
	exponents $\alpha_\varphi$
	and 
	$\alpha_\psi$, respectively,
	and with the same bound 
	$K$.
	If 
	\begin{equation*}
	\sup_{n\in \mathbb{Z}}
	\{K(\omega_p)\|\varphi_n(\omega_p)-\psi_n(\omega_p)\|_{{\mathcal L}(X)}\}\leq
	\epsilon , 
	\end{equation*}
	then
	\begin{equation*}
	\sup_{n\in \mathbb{Z}}  \|\Pi_\varphi^s(\Theta_n\omega_p)-\Pi_\psi^s(\Theta_n\omega_p)\|_{{\mathcal L}(X)}\leq \frac{e^{-\alpha_\psi(\omega_p)} + e^{-\alpha_\varphi(\omega_p)}}{1-e^{-(\alpha_\psi(\omega_p)+\alpha_\varphi(\omega_p))}}\, \epsilon.
	\end{equation*}
\end{theorem}

\begin{proof}
	Let $z\in X$, $\omega\in \tilde{\Omega}$, and $m,p\in \mathbb{Z}$ be fixed and consider
	\begin{equation*}
	f_n(\omega_p)= \left\{ 
	\begin{array}{l l} 
	0,
	& \quad \hbox{if } n\neq m-1,
	\\ z, \, 
	& \quad \hbox{if } n=m-1.
	\end{array} 
	\right.
	\end{equation*} 
	Thus by Theorem \ref{th-admissibility-for-pertubed-process} for each
	$\omega_p\in \mathbb{Z}\times\tilde{\Omega}$
	there exists a bounded solution $x(\omega_p)=\{x_n(\omega_p)\}_{n\in\mathbb{Z}}$ given by
	$x_n^{j}(\omega_p):=G_{\omega_p}^{j}(n,m) z^{-1}$ for $j=\varphi,\psi$. 
	Note that
	\begin{equation*}
	x_{n+1}^{\varphi} - \psi_n(\omega_p) x_n^{\varphi}=
	\varphi_n(\omega_p) x_n^{\varphi}-\psi_n(\omega_p) x_n^{\varphi}+f_n(\omega_p)  
	\end{equation*}
	and
	$	x_{n+1}^{\psi} - \psi_n(\omega_p) x_n^{\psi}=f_n(\omega_p)$.
	Then, if
	$z_n:=x_n^{\varphi}-x_n^{\psi}$ we obtain that  
	$z_{n+1}=\varphi_n z_n + y_n$,
	where
	$y_n:=(\varphi_n(\omega_p)-\psi_n(\omega_p)x_n^{\varphi}(\omega_p)$ for all $n\in \mathbb{Z}$. 
	Thanks to the boundedness of the sequence $\{x_n^{\varphi}(\omega_p)\}_{n\in\mathbb{Z}}$ 
	and by the hypotheses on $\varphi_n-\psi_n$ we have that 
	$\{y_n\}_{n\in {\mathbb Z}}$ is bounded and by 
	Theorem \ref{th-admissibility-for-pertubed-process} 
	we have that
	\begin{equation*}
	z_n=\sum_{k=-\infty}^\infty G_{\omega_p}^{\psi}(n,k+1)(\omega_p)
	(\varphi_k(\omega_p)-\psi_k(\omega_p))G_{\omega_p}^{\varphi}(k,m) z
	\end{equation*}
	and therefore, by the hypotheses on 
	$\Psi-\Phi$,
	we deduce
	\begin{eqnarray*}
		\|z_m\|_X&\leq&\sum_{k=-\infty}^\infty K(\omega_p) e^{-\alpha_\psi(\omega_p)|m-k-1|}
		e^ {-\alpha_\psi(\omega_p)|k-m|} \|\varphi_k(\omega_p)-\psi_k(\omega_p)\|_{{\mathcal
				L}(X)}\|z\|_{X} \\
		&\leq& \frac{e^{-\alpha_\psi(\omega_p)} + e^{-\alpha_\varphi(\omega_p)}}{1-e^{-(\alpha_\psi(\omega_p)+\alpha_\varphi(\omega_p))}}\, \epsilon \,
		\|z\|_X.
	\end{eqnarray*}
	The definition of $z$ in $m$ yields
	$$
	z_m=x_m^{\varphi}-x_m^{\psi}=(G_{\omega_p}^{\varphi}(m,m)-G_{\omega_p}^{\psi}(m,m))z=
	(\Pi_\varphi^s(\Theta_m\omega_p)-\Pi_\psi^s(\Theta_m\omega_p))z.
	$$
	Consequently,
	\begin{equation*}
	\|(\Pi_\varphi^s(\Theta_m\omega_p)-\Pi_\psi^s(\Theta_m\omega_p))z\|_X\leq \frac{e^{-\alpha_\psi(\omega_p)} + e^{-\alpha_\varphi(\omega_p)}}{1-e^{-(\alpha_\psi(\omega_p)+\alpha_\varphi(\omega_p))}}\, \epsilon \,
	\|z\|_X,
	\end{equation*}
	which concludes the proof of the theorem.
\end{proof}

\subsection{Exponential dichotomy for co-cycle: continuous case}\label{subsec-ed-continuous}
\par We study exponential dichotomies for a 
continuous nonautonomous random dynamical system 
$(\varphi,\Theta)$. 
\par Our goal is to prove a robustness result for nonautonomous random dynamical systems
that possesses a  
uniform exponential dichotomy. This section follows closely the ideas of Chow and Leiva in
\cite{Chow-Leiva-existence-roughness}. However, while they consider a driving flow 
$\theta_t:\Sigma\rightarrow\Sigma$ on a compact Hausdorff space $\Sigma$ in a deterministic context, we deal with a nonautonomous random dynamical systems driven by a flow $\mathbb{R}\times \Omega\ni (\tau,\omega)\mapsto \Theta_t(\tau,\omega)=(t+\tau,\theta_t\omega)\in \mathbb{R}\times \Omega$, where $\theta_t:\Omega\to \Omega$ is a random flow defined on a probability space $\Omega$.

\par We first prove that if a co-cycle possesses an exponential dichotomy, then 
its discretization also admits an exponential dichotomy.

\begin{theorem}\label{th-continousED-implies-discreteED}
	Let $(\varphi,\Theta)_{(X,\mathbb{R}\times\Omega)}$ be a linear co-cycle that admits an exponential dichotomy with bound $K$, exponent $\alpha$ and family of projections $\Pi^u:=\{\Pi^u(\omega_p): \omega_p\in \mathbb{R}\times \widetilde{\Omega}\}$, where $\mathbb{P}(\widetilde{\Omega})=1$ and $\widetilde{\Omega}$ is $\theta$-invariant.
	Then
	for each $\omega_p\in \mathbb{R}\times \widetilde{\Omega}$ the sequence of linear operators 
	$\{\varphi_n(\omega_p):=\varphi(1,\Theta_n\omega_p)\, ; \, n\in \mathbb{Z}\}$ 
	admits an exponential dichotomy with bound 
	$K(\omega_p)$ and 
	exponent $\alpha(\omega_p)$.
\end{theorem}
\begin{proof}
	Let $\Pi^s$ be a family of projections associated with the
	exponential dichotomy.
	Define, for each $n\in \mathbb{Z}$ and 
	$\omega_p$, the projector $P_n(\omega_p):=\Pi^s(\Theta_n\omega_p)$. Then
	\begin{eqnarray*}
		P_{n+1}(\omega_p)\varphi_n(\omega_p)&=&
		\Pi^s(\Theta_{n+1}(\omega_p))\varphi(1,\Theta_n(\omega_p))\\
		&=&\varphi(1,\Theta_n(\omega_p))\Pi^s(\Theta_{n}(\omega_p))\\
		&=&\varphi_n(\omega_p)P_{n}(\omega_p),
	\end{eqnarray*}
	and the first property is proved.
	Note that, if $Q=Id_X-P$, we have that
	\begin{equation*}
	\varphi_n(\omega_p)|_{R(Q_n(\omega_p))  )}
	=\varphi(1,\Theta_n(\omega_p))|_{R(\Pi^u(\Theta_{n}(\omega_p))  )}
	\end{equation*}
	is an isomorphism onto $R(Q_{n+1}(\omega_p)) $.
	\par Finally, for $n\geq m$ we see that
	\begin{equation*}
	\|\varphi_{n,m}(\omega_p)P_m(\omega_p)\|_{\mathcal{L}(X)}\leq K(\Theta_m\omega_p)
	e^{-\alpha(\omega_p)(n-m)},
	\end{equation*}
	and for $n<m$
	\begin{equation*}
	\|\varphi_{n,m}(\omega_p)Q_m(\omega_p)\|_{\mathcal{L}(X)}\leq K(\Theta_m\omega_p)
	e^{\alpha(\omega_p)(n-m)},
	\end{equation*}
	where $\varphi_{n,m}(\omega_p)$ is the inverse of $\varphi_{m,n}(\omega_p)$ over 
	$R(Q_m(\omega_p))$, and the proof follows by the $\Theta$-invariance property of $K$.
\end{proof}

\par As a corollary of Theorem \ref{th-continousED-implies-discreteED} and 
Corollary \ref{cor-uniqueness-projection-discrete} we obtain uniqueness of projectors for the continuous case.

\begin{corollary}\label{cor-uniqueness-projection-continuous}
	If $(\varphi,\Theta)$ admits an exponential dichotomy, then
	the family of projections are uniquely determined.
\end{corollary}

\par Now, we provide conditions to prove a kind of converse result of Theorem
\ref{th-continousED-implies-discreteED}. If the discretization admits an exponential dichotomy then the continuous co-cycle also possesses.

\begin{theorem}\label{th-discrete-dichotomy-implies-continuous-dichotomy}
	Let $(\varphi,\Theta)_{(X,\mathbb{R}\times\Omega)}$ be
	a co-cycle and for each $\omega_p$ consider the 
	associated sequence of operators 
	\begin{equation*}
	\{\varphi_n(\omega_p):=\varphi(1,\Theta_n\omega_p)\}_{n\in \mathbb{Z}}.
	\end{equation*}
	Suppose that there is a full measure set 
	$\tilde{\Omega}$ such that for each $\omega_p\in \mathbb{R}\times \tilde{\Omega}$ 
	\begin{itemize}
		\item we have that
		\begin{equation*}
		L(\omega_p):=\sup_{0\leq t\leq 1} \|\varphi(t,\omega_p)\|_{\mathcal{L}(X)} < +\infty,
		\end{equation*}
		satisfies $L(\Theta_t\omega_p)\leq L(\omega_p)$, for all $t\in \mathbb{R}$.
		\item there exists $\Theta$-invariant maps
		$K,\alpha$ such that the sequence 
		$\{\varphi_n(\omega_p)\}_{n\in \mathbb{Z}}$ admits an
		exponential dichotomy with bound $K(\omega_p)$,
		exponent $\alpha(\omega_p)$, and family of projections
		$\{P_n(\omega_p): n\in \mathbb{Z}\}$ such that
		for each $(n,p)\in \mathbb{Z}\times\mathbb{R}$ 
		the map
		$P_n(p,\cdot):\tilde{\Omega}\rightarrow \mathcal{L}(X)$ is strongly measurable.
	\end{itemize}
	Then $(\varphi,\Theta)$ admits an exponential dichotomy with exponent
	$\alpha$, and bound
	\begin{equation*}
	\hat{K}(\omega_p)=K(\omega_p)\sup_{0\leq t\leq  1}
	\{\|\varphi(t,\omega_p)\|_{\mathcal{L}(X)}e^{\alpha(\omega_p)t}\,\}.
	\end{equation*}
\end{theorem}

\begin{proof}
	Let $\{P_n(\omega_p); n\in \mathbb{Z} \}$ be the family of projectors associated with the 
	exponential dichotomy of $\{\varphi_n(\omega_p)\}_{n\in \mathbb{Z}}$.
	Define $\Pi^s:\mathbb{R}\times \tilde{\Omega}\to \mathcal{L}(X)$ by
	\begin{equation*}
	\Pi^s(\omega_p):=P_0(\omega_p).
	\end{equation*}
	Thus for each $p\in \mathbb{R}$ fixed 
	$\Pi^s(p,\cdot):\tilde{\Omega}\to \mathcal{L}(X)$ is strongly measurable.
	\par \textbf{Claim 1:} For each $k\in \mathbb{Z}$ fixed,
	we have that
	$P_{k}(\omega_p)=\Pi^s(\Theta_k\omega_p)$.
	\par Indeed, for each 
	$k\in \mathbb{Z}$ fixed the sequence
	$\{\varphi_n(\Theta_k\omega_p)\}_{n\in \mathbb{Z}}$ admits 
	an exponential dichotomy with projections
	$\{P_n(\Theta_k\omega_p); n\in \mathbb{Z} \}$. 
	Note that,
	\begin{equation*}
	\varphi_n(\Theta_k\omega_p)=\varphi(1,\Theta_n(\Theta_k\omega_p))=
	\varphi_{n+k}(\omega_p).
	\end{equation*}
	Then, from Lemma \ref{cor-uniqueness-projection-discrete} we have that
	$P_n(\Theta_k\omega_p)=P_{n+k}(\omega_p)$ for all 
	$n,k\in \mathbb{Z}$. In particular,
	$P_k(\omega_p)=\Pi^s(\Theta_{k}\omega_p)$.
	\par Next, we prove that this projector operator is the candidate to obtain the exponential dichotomy.
	\par \textbf{Claim 2:} For all $t\geq 0$ and $\omega_p\in \mathbb{Z}\times\tilde{\Omega}$, 
	we have that
	\begin{equation*}
	\|\varphi(t,\omega_p)\Pi^s(\omega_p)\|_{\mathcal{L}(X)} \leq \hat{K}(\omega_p) 
	e^{-\alpha(\omega_p)t},
	\end{equation*}	
	where $\hat{K}(\omega_p)=K(\omega_p)\sup_{0\leq t\leq 1}
	\{e^{\alpha(\omega_p)t} \|\varphi(t,\omega_p)\|_{\mathcal{L}(X)}\}$.
	\par Indeed, choose $n\in \mathbb{N}$, such that 
	$n\leq t<n+1$, then
	we write 
	\begin{equation*}
	\varphi(t,\omega_p)=\varphi(t-n,\Theta_n\omega_p)\varphi(n,\omega_p).
	\end{equation*}
	Therefore
	\begin{eqnarray*}
		\|\varphi(t,\omega_p)\Pi^s(\omega_p)\|_{\mathcal{L}(X)}
		&\leq&
		K(\omega_p) e^{-\alpha(\omega_p) n} \|\varphi(t-n,\Theta_n\omega_p)\|_{\mathcal{L}(X)}\\
		&\leq& \hat{K}(\omega_p) e^{-\alpha(\omega_p) t}.
	\end{eqnarray*}
	
	
	\par\textbf{Claim 3:} Let $x\in R(\Pi^u(\omega_p))$, $t<0$ and choose 
	$n\leq 0$
	such that 
	$n\leq t<n+1$. Define the linear operator
	\begin{equation*}
	\varphi(t,\omega_p)x:=\varphi(t-n,\Theta_n\omega_p)\varphi(n,\omega_p)x,
	\end{equation*}
	where $\varphi(n,\omega_p)$ is the inverse of
	$\varphi(-n,\Theta_n\omega_p)|_{R(\Pi^u(\Theta_n\omega_p))}$.
	Then for all $t\leq 0$
	\begin{equation*}
	\|\varphi(t,\omega_p)\Pi^u(\omega_p)\|_{\mathcal{L}(X)}
	\leq \hat{K}(\omega_p) e^{\alpha(\omega_p) t}.
	\end{equation*}
	\par The proof of Claim 3 follows by a similar argument used on the proof of Claim 2.
	\par \textbf{Claim 4:} The range of $\Pi^s(\omega_p)$ is characterized as follows
	\begin{equation*}
	R(\Pi^s(\omega_p))=\{z\in X\, ; [0,+\infty) \ni t \mapsto \varphi(t,\omega_p)z 
	\hbox{ is bounded } \}.
	\end{equation*}
	\par Indeed, if $x\in R(\Pi^s(\omega_p))$ from Claim 2 we have that
	\begin{equation*}
	\|\varphi(t,\omega_p)x\|_X
	\leq \hat{K}(\omega_p) e^{-\alpha(\omega_p) t}\|x\|_X, \hbox{ for every } t\geq 0.
	\end{equation*}
	Thus, it follows that $[0,+\infty)\ni t\mapsto \varphi(t,\omega_p)x$ is bounded.
	Conversely, suppose that $x\notin R(\Pi^s(\omega_p))$ and define
	$v=\varphi(n,\omega_p)\Pi^u(\omega_p)x,$, 
	hence
	\begin{equation*}
	\|\varphi(-n,\Theta_n\omega_p)\Pi^u(\Theta_n\omega_p)v\|_X\leq 
	K(\Theta_n\omega_p)e^{-\alpha(\omega_p)n}\|v\|_X, 
	\end{equation*}
	for $n\geq 0$.
	Thus we obtain
	\begin{equation*}
	\|\Pi^u(\omega_p)x\|_X\leq K(\omega_p) e^{\alpha(\omega_p)n} 
	\|\varphi(n,\omega_p)\Pi^u(\omega_p)x\|_X, \hbox{ for } n\geq 0.
	\end{equation*}
	Since $\Pi^u(\omega_p)x\neq0$ and 
	we obtain that
	$[0,+\infty) \ni n \mapsto \varphi(n,\omega_p)\Pi^u(\omega_p)x$ is unbounded,
	then the mapping
	$[0,+\infty) \ni t \mapsto \varphi(n,\omega_p)x$ is unbounded and complete the proof of Claim 4.
	\par \textbf{Claim 5:} The range of $\Pi^u(\omega_p)$ is characterized as follows: $z\in \Pi^u(\omega_p)$ if and only if there exists a backwards bounded solution $\xi^*$ for the evolution process $\Phi_{\omega_p}=\{\varphi_{t,s}(\omega_p): t\geq s\}$, see Remark \ref{remark-associated-processes}, such that $\xi^*(0)=z$.
	\par In fact, let $z\in R(\Pi^u(\omega_p))$, and $t<0$, 
	we define 
	\begin{equation*}
	\xi(t)=\varphi(t-n,\Theta_n\omega_p)\varphi(n,\omega_p)z,
	\end{equation*}
	where $n\leq 0$ is such that $n\leq t< n+1$.
	Then $\xi$ is a backwards bounded solution for
	the evolution process 
	$\{\varphi_{t,s}(\omega_p):t\geq s\}$ through $z$.
	In fact, for $t\geq s$
	\begin{eqnarray*}
		\varphi_{t,s}(\omega_p)\xi(s)&=&
		\varphi(t-s,\Theta_s\omega_p)\varphi(s-n,\Theta_n\omega_p)\varphi(n,\omega_p)z\\
		&=&\varphi(t-n,\Theta_n\omega_p)z=\xi(t),
	\end{eqnarray*}
	which shows that $\xi$ is a backwards solution.
	From Claim 3 we have that $\xi$ is bounded and 
	$\xi(0)=z$.
	Conversely, choose $x\in X$ 
	such that $x\notin R(\Pi^u(\omega_p))$ and
	suppose that there exists a backwards solution of
	$\{\varphi_{t,s}(\omega_p): t\geq s\}$ through $x$ on $t=0$.
	Then, for $n\leq 0$,
	\begin{eqnarray*}
		\|\Pi^s(\omega_p)x\|_X&=&
		\|\Pi^s(\omega_p)\varphi(-n,\Theta_n\omega_p)\xi(n)\|_X\\
		&\leq&	\|\varphi(-n,\Theta_n\omega_p)\Pi^s(\Theta_n\omega_p)\|_{\mathcal{L}(X)}
		\|\xi(n)\|_X\\
		&\leq& K(\omega_p)e^{\alpha(\omega_p)n}\|\xi(n)\|_X.
	\end{eqnarray*}
	Since $\Pi^s(\omega_p)x \neq 0$ we see that 
	$\xi$ is unbounded, and the proof is complete.
	
	\par Now, it follows as in the discrete case, see Lemma \ref{lemma-key-to-prove-robstness-of-discrete}, that: 
	\par \textbf{Claim 6:} $R(\Pi^s(\cdot))$ is positively invariant and $R(\Pi^u(\omega_p)$ is invariant, i.e.,
	\begin{eqnarray*}
		\varphi(t,\omega_p)R(\Pi^s(\omega_p))&\subset&
		R(\Pi^s(\Theta_{t}\omega_p)), \hbox{ for all }t\geq 0, \hbox{ and }\\
		\varphi(t,\omega_p) R(\Pi^u(\omega_p))&=&R(\Pi^u(\Theta_t\omega_p)), 
		\hbox{ for all } t\geq 0.
	\end{eqnarray*}

	\par \textbf{Claim 7:} The linear operator
	$\varphi(t,\omega_p): R(\Pi^u(\omega_p))\rightarrow X$
	is injective.
	\par Indeed, let $z\in R(\Pi^u(\omega_p)$ such that $\varphi(t,\omega_p)z=0$.
	Choose $n\in \mathbb{N}$ such that $n\leq t\leq n+1$, then
	\begin{equation*}
	0=\varphi(n-t,\Theta_t\omega_p)\varphi(t,\omega_p)z=\varphi(n,\omega_p)z.
	\end{equation*}
	Now, Claim 7 follows by the fact that $\varphi(n,\omega_p)|_{R(\Pi^u(\omega_p))}$
	is injective for any
	integer $n\leq 0$.
	\par Then it follows directly from Claims 6 and 7 that
	$\varphi(t,\omega_p):R(\Pi^u(\omega_p))\rightarrow R(\Pi^u(\Theta_t\omega_p))$ is an isomorphism.
	\par Finally, from Claim 6, we obtain that 
	$\varphi(t,\omega_p)\Pi^s(\omega_p)=
	\Pi^s(\Theta_t\omega_p)\varphi(t,\omega_p)$ for all
	$t\geq 0$, and the proof of the theorem is complete.
\end{proof}

\par Now, we state our robustness result for nonautonomous random dynamical systems with an exponential dichotomy.
\begin{theorem}\label{th-robustness-uniform-ED-co-cycle}
	Let $(\varphi,\Theta)$ be an co-cycle with an exponential dichotomy with bound 
	$K$ and exponent $\alpha$. Assume that there is a random variable
	$L:\mathbb{R}\times\Omega\rightarrow (0,+\infty)$ such that
	\begin{equation*}
	L(\omega_p):=\sup_{0\leq t\leq  1}\Big\{\|\varphi(t,\omega_p)\|_{\mathcal{L}(X)}\Big\}<+\infty,
	\end{equation*}
	that satisfies $L(\Theta_t\omega_p)\leq L(\omega_p)$, for all $t\in \mathbb{R}$.
	Then there exists a $\Theta$-invariant map $\delta:\mathbb{Z}\times \Omega\to \mathbb{R}$ with
	\begin{equation*}
	0< \delta(\omega_p) < 
	\frac{1-e^{-\alpha(\omega_p)}}{1+e^{-\alpha(\omega_p)}},
	\hbox{ for each }\omega_p\in \mathbb{R}\times\Omega,
	\end{equation*}
	such that every co-cycle $(\psi,\Theta)$ satisfying
	\begin{equation*}
	\sup_{0\leq t\leq  1}\Big\{\|\varphi(t,\omega_p)-\psi(t,\omega_p)\|_{\mathcal{L}(X)}\Big\}\leq \delta(\omega_p),
	\end{equation*}
	admits an exponential dichotomy with
	exponent $\tilde{\alpha}(\omega_p)$ and bound
	\begin{equation*}
	\hat{M}(\omega_p)=M(\omega_p)\sup_{0\leq t\leq  1}
	\Big\{\|\psi(t,\omega_p)\|_{\mathcal{L}(X)}e^{\tilde{\alpha}(\omega_p)t}\,\Big\},
	\end{equation*}
	where $M$ and $\tilde{\alpha}$ are the bound and exponent of the discretization of
	$(\psi,\Theta)$ given in Theorem 
	\ref{th-discrete-robstness-random-perturbations-of-autonomous-nonuniform-case}.
\end{theorem}

\begin{proof}
	First, we consider the discretization of the co-cycle 
	$(\varphi,\Theta)$, i.e., for each $\omega_p$ we consider 
	the family of linear operators 
	$\{\varphi_n(\omega_p):=\varphi(1,\omega_p): n\in \mathbb{Z} \}$.
	From
	Theorem \ref{th-continousED-implies-discreteED}, we have that
	$\{\varphi_n(\omega_p) \, : \, n\in \mathbb{Z}\}$ admits a exponential dichotomy with bound 
	$K(\omega_p)$ and exponent $\alpha(\omega_p)$.
	By Theorem \ref{th-discrete-robstness-random-perturbations-of-autonomous-nonuniform-case}, there exists a
	$\Theta$-invariant map 
	$\delta$ such that
	if
	$\{\psi_n(\omega_p)\}_{n\in \mathbb{Z}}$ is a sequence of bounded linear operators which satisfies
	\begin{equation*}
	\sup_{n\in \mathbb{Z}}\|\varphi_n(\omega_p)-\psi_n(\omega_p)\|_{\mathcal{L}(X)}\leq \delta(\Theta_n\omega_p)=\delta(\omega_p),
	\end{equation*}
	$\{\psi_n(\omega_p)\}_{n\in \mathbb{Z}}$ admits an exponential dichotomy with 
	bound $M(\omega_p)$
	and exponent $\tilde{\alpha}(\omega_p)$ (see Theorem \ref{th-discrete-robstness-random-perturbations-of-autonomous-nonuniform-case}).
	Now, in order to use Theorem \ref{th-discrete-dichotomy-implies-continuous-dichotomy}
	to guarantee that $(\psi,\Theta)$ admits an exponential dichotomy
	it remains only to see that
	\begin{equation*}
	\sup_{0\leq t\leq  1}\|\psi(t,\omega_p)\|_{\mathcal{L}(X)}\leq 
	\delta(\omega_p)+ L(\omega_p)<+\infty.
	\end{equation*}
	Therefore, the hypotheses of Theorem 
	\ref{th-discrete-dichotomy-implies-continuous-dichotomy}
	are satisfied, and the proof is complete.
\end{proof}

\begin{remark}
	Note that for each nonautonomous evolution process $(\varphi,\Theta)$ with 
	a uniform exponential dichotomy there exists $\delta$ in the previous theorem that depends only
	on the exponent of exponential dichotomy. 
	When applying Theorem \ref{th-discrete-robstness-random-perturbations-of-autonomous-nonuniform-case} we obtain explicit functions for the bound and exponent of the perturbations, which is an improvement of the result of robustness in the case of $\Omega$ a Hausdorff compact topological space on \cite{Chow-Leiva-existence-roughness}.
\end{remark}

\par To end this subsection we extend the result on the continuous dependence of projections for discrete co-cycle, Theorem \ref{th-continuity-projection-discrete}, to continuous co-cycle.

\begin{theorem}\label{th-continuity-projection-continuous}
	Suppose that  $(\varphi,\Theta)$ and $(\psi,\Theta)$ are nonautonomous random dynamical systems and that they
	admit an exponential dichotomy with projections 
	$\Pi_\varphi^s$ and $\Pi_\psi^s$, and
	exponents $\alpha_\varphi$
	and 
	$\alpha_\psi$, respectively.
	If 
	\begin{equation*}
	\sup_{t\in \mathbb{R}}
	\Big\{K(\omega_p)\|\varphi(t,\omega_p)-\psi(t,\omega_p)\|_{{\mathcal L}(X)}\Big\}\leq
	\epsilon , 
	\end{equation*}
	then
	$$
	\sup_{t\in \mathbb{R}}  \|\Pi_\varphi^s(\Theta_t\omega_p)-\Pi_\psi^s(\Theta_t\omega_p)\|_{{\mathcal L}(X)}\leq \frac{e^{-\alpha_\psi(\omega_p)} + e^{-\alpha_\varphi(\omega_p)}}{1-e^{-(\alpha_\psi(\omega_p)+\alpha_\varphi(\omega_p))}}\, \epsilon.
	$$
\end{theorem}

\section{Hyperbolicity on nonautonomous random differential equations}\label{sec-applications}

\par This section is dedicated to apply the exponential dichotomies results to nonautonomous 
random/stochastic differential equations. First we prove the existence of an exponential dichotomy for a linear nonautonomous random dynamical system obtained from a nonautonomous random perturbation of an autonomous problem. Then we apply this result to show existence and continuity of random hyperbolic solutions for nonautonomous random semilinear differential equations.
\par Before we start we remark some facts about nonautonomous random differential equations and generation of nonautonomous random dynamical systems.	
Let $(\theta,\Omega,\mathcal{F},\mathbb{P})$ be a random flow, and
consider the following  initial value problem
\begin{equation}\label{eq-generation-co-cycle}
\dot{y}=f(t,\theta_t\omega,y), \ \ t>\tau\hbox{ and } y(\tau)=y_0.
\end{equation}
\par Assume that for almost all $\omega\in \Omega$ the solutions of \eqref{eq-generation-co-cycle}
are associated with a nonlinear evolution process $\mathcal{S}_\omega:=\{S_\omega(t,s): t\geq s\}$. More precisely, if for every $\tau \in \mathbb{R}$ and $y_0\in X$ 
there exists $[\tau,+\infty)\ni t\mapsto y(t,\tau,\omega;y_0)$ a solution for
\eqref{eq-generation-co-cycle}, then we define 
$S_\omega(t,\tau,\omega)u_0:=y(t,\tau,\omega;u_0)$.
\par Another equivalent way to generate a dynamical system from problem \eqref{eq-generation-co-cycle} is the following:
define $x(t):=y(t+\tau,\tau,\omega;u_0)$, for every $t\geq 0$ and some fixed $\tau\in \mathbb{R}$. Hence we obtain the initial value problem 
\begin{equation*}\label{eq-generation-co-cycle-for-x}
	\dot{x}=f(t+\tau,\theta_{t+\tau}\omega,x)\ \ t>0\hbox{ and } x(0)=x_0.
\end{equation*}
Now, the relation
$f(t+\tau,\theta_{t+\tau}\omega)=f(\Theta_t(t,\theta_\tau\omega))$,
where $\Theta$ is the flow
$\{\Theta_t:\mathbb{R}\times\Omega\to \mathbb{R}\times\Omega\}_{t\in \mathbb{R}}$ 
defined as $\Theta_t(\tau,\omega)=(t+\tau,\theta_t\omega)$, 
leads us to consider a
nonautonomous random differential with a nonlinearity 
driven by the flow $\Theta$ i.e.
\begin{equation}\label{eq-generation-co-cycle-for-z}
\dot{z}=f(\Theta_t(\tau,\omega),z), \ \ t>0\hbox{ and } z(0)=z_0, \hbox{ for each }
(\tau,\omega)\in \mathbb{R}\times \Omega.
\end{equation}
Thus, the solutions of \eqref{eq-generation-co-cycle-for-z} defines a 
nonautonomous random dynamical system 
$\varphi(t,\tau,\omega)z_0:= z(t,(\tau,\omega);z_0)$.
\par Therefore, we rewrite problem \eqref{eq-generation-co-cycle} using formulation of \eqref{eq-generation-co-cycle-for-z} as follows
\begin{equation*}
\varphi(t,\tau,\omega)y_0:=y(t+\tau,\tau,\theta_{-\tau}\omega;y_0).
\end{equation*}
Or equivalently, for almost all $\omega\in \Omega$ we have that
\begin{equation*}
\varphi(t,\tau,\omega)=S_{\theta_{-\tau}\omega}(t+\tau,\tau), \ t\geq 0, \ \tau\in \mathbb{R}.
\end{equation*}
Consequently, we study asymptotic behavior of both dynamical systems: the co-cycle 
$(\varphi,\Theta)_{(X,\mathbb{R}\times \Omega)}$ generated by
\eqref{eq-generation-co-cycle-for-z}, and the family of evolution processes 
$\{\mathcal{S}_\omega;\, \omega\in \Omega\}$ associated with
\eqref{eq-generation-co-cycle}. 
%

\subsection{Linear nonautonomous random differential equations}\label{subsec-app-linear}
\quad
\par In this subsection we shall study linear nonautonomous random differential equations on a Banach space $X$. We provide conditions to guarantee the existence of an exponential dichotomy for a nonautonomous random perturbation of a hyperbolic autonomous problem. The results contained here were inspired by those of Chow and Leiva \cite{Chow-Leiva-existence-roughness} in a deterministic context, where the base flow is a group over a Hausdorff compact set.
\par Let $A$ be the generator of a strongly continuous semigroup
$\{e^{At}:\, t\geq 0\}$ and $B:\mathbb{R}\times\Omega\rightarrow \mathcal{L}(X)$ be a bounded operator depending on parameters on $\mathbb{R}\times\Omega$. We study the linear problem
\begin{equation}\label{eq-autonomous-perturbation-by-nonautonomous-random}
\dot{x}=Ax+B(\Theta_t\omega_\tau)x, \ \ t>0\hbox{ and } x(0)=x_0.
\end{equation}
where $\omega_\tau:=(\tau,\omega)$ and for every $t\in \mathbb{R}$ the map
$\Theta_t:\mathbb{R}\times\Omega \to \mathbb{R}\times\Omega$ is defined by 
$\Theta_t\omega_\tau:=(t+\tau,\theta_t\omega)$.

To study equation \eqref{eq-autonomous-perturbation-by-nonautonomous-random}, we consider the following family of integral equations
\begin{equation*}
x(t,\tau,\omega;x_0)=e^{At}x_0+\int_{0}^{t}e^{A(t-s)}B(\Theta_s\omega_\tau)x(s)ds,\  x_0\in X, \ 
t\geq 0, \ \omega_\tau\in \mathbb{R}\times \Omega.
\end{equation*}
\par We have the following result on the robustness of exponential dichotomies for 
linear nonautonomous random differential equations.
%

\begin{theorem}\label{th-robustness-nonautonomous-random-perturbation-of-nonautonomous}
	Let $(\varphi,\Theta)$ be a linear nonautonomous random dynamical system with
	\begin{equation}
	L(\omega_\tau):=\sup_{t\in \mathbb{R}}\|\varphi(t,\omega_\tau)\|_{\mathcal{L}(X)}<+\infty, 
	\ \hbox{ for each } \omega_\tau\in \mathbb{R}\times \Omega,
	\end{equation}
	where $\sup_{t\in\mathbb{R}}L(\Theta_t\omega_\tau)\leq  L(\omega_\tau)$.
	Suppose that $(\varphi,\Theta)$ admits an exponential dichotomy 
	with exponent $\alpha$ and bound $K$. 
	Then there exists a $\Theta$-invariant map $\epsilon:\mathbb{R}\times\Omega\to (0,+\infty)$
	such that for every 
	$B:\mathbb{R}\times\Omega\rightarrow \mathcal{L}(X)$ 
 	with 
	\begin{equation*}
	\sup_{0\leq t\leq  1}\|\int_0^tB(\Theta_s\omega_\tau)xds\|_X<
	\epsilon(\omega_\tau)\|x\|_X
	\end{equation*} 
	a nonautonomous random 
	dynamical system satisfying
	\begin{equation}\label{eq-th-robustness-nonautonomous-random-perturbation-of-nonautonomous}
	\psi(t,\omega_\tau)=\varphi(t,\omega_\tau)+\int_0^t\varphi(t-s,\Theta_s\omega_{\tau})
	B(\Theta_s\omega_\tau)\psi(s,\omega_\tau) ds
	\end{equation}
	admits an exponential dichotomy with bound $\hat{M}$ and exponent $\tilde{\alpha}$ provided in Theorem
	\ref{th-robustness-uniform-ED-co-cycle}.
\end{theorem}
\begin{proof}
	Let $(\psi,\Theta)$ be a nonautonomous random dynamical system satisfying \eqref{eq-th-robustness-nonautonomous-random-perturbation-of-nonautonomous}.
	Then
	\begin{equation*}
	\|\psi(t,\omega_\tau)x\|_X\leq L(\omega_\tau)\|x\|_X+\int_{0}^{t}L(\Theta_s\omega_\tau)\|B(\Theta_s\omega_\tau)\|_{\mathcal{L}(X)} \|\psi(s,\omega_\tau)x\|_X \, ds, \ 0\leq t\leq 1.
	\end{equation*}
	From Gr\"onwall's inequality we obtain
	$\|\psi(t,\omega_\tau)\|_{\mathcal{L}(X)}\leq L_1(\omega_\tau):= L(\omega_\tau)e^{L(\omega_\tau)\epsilon(\omega_\tau)}$, for every $0\leq t\leq 1$.
	\par Hence, for every $0\leq t\leq 1$ 
	\begin{equation*}
	\|\varphi(t,\omega_\tau)x-\psi(t,\omega_\tau)x\|_X 
	\leq\epsilon(\omega_\tau) L(\omega_\tau)L_1(\omega_\tau)\|x\|_X.
	\end{equation*}
	Finally, since $(\varphi,\Theta)$ admits an exponential dichotomy, there exists a 
	$\Theta$-invariant measurable map $\delta>0$ as in Theorem
	\ref{th-robustness-uniform-ED-co-cycle}. Therefore, for each $\omega_\tau$ choose $\epsilon=\epsilon(\omega_\tau)>0$ such that 
	$\epsilon(\omega_\tau) L(\omega_\tau)
	L_1(\omega_\tau)<\delta(\omega_\tau)$. Note that, for every $t\in \mathbb{R}$ 
	$L_2(\Theta_t\omega_\tau):= L(\Theta_t\omega_\tau)
	L_1(\Theta_t\omega_\tau)\leq L_2(\omega_\tau)$, therefore we choose $\epsilon(\Theta_t\omega_\tau)=\epsilon(\omega_\tau)$, and the proof is complete.
\end{proof}

\begin{remark}
	Theorem \ref{th-robustness-nonautonomous-random-perturbation-of-nonautonomous} is saying that if problem 
	$\dot{x}=A(\Theta_t\omega_\tau)x$
	generates a nonautonomous random dynamical system with an exponential dichotomy then
	for the class of bounded linear perturbation $B$ given in the Theorem the perturbed nonautonomous random dynamical system generated by problem 
	\begin{equation}\label{eq-remark-robustness-nonautonomous-random-perturbation-of-nonautonomous}
	\dot{x}=A(\Theta_t\omega_\tau)x+B(\Theta_t\omega_\tau)x,
	\ x(0)=x_0\in X, \ 
	t\geq 0,
	\end{equation}
	admits an exponential dichotomy.
\end{remark}
\par Now, as a corollary we have the following robustness result for an autonomous problem under nonautonomous random perturbation.

\begin{theorem}\label{th-robustness-nonautonomous-random-perturbation-of-autonomous}
	Assume that $\mathcal{A}$ generates a \textbf{analytic semigroup} $\{e^{\mathcal{A}t}: t\geq 0\}$, and that
	the spectrum of $\mathcal{A}$, $\sigma(\mathcal{A})$ does not intersect the imaginary 
	axis and that the set $\sigma^+:=\{\lambda\in \sigma(\mathcal{A}):\hbox{ Re}\lambda>0\}$ is compact.
	Then there exists a $\Theta$-invariant map $\epsilon:\mathbb{R}\times\Omega\to (0,+\infty)$
	such that, if 
	$B:\mathbb{R}\times\Omega\rightarrow \mathcal{L}(X)$ 
	satisfies
	\begin{equation*}
	\sup_{0\leq t\leq  1}\|\int_0^tB(\Theta_s\omega_\tau)xds\|_X<
	\epsilon(\omega_\tau)\|x\|_X,
	\end{equation*} 
	then any nonautonomous random dynamical system satisfying
	\begin{equation}\label{eq-existence-nonautonomous-random-perturbation}
	\varphi(t,\omega_\tau)=e^{\mathcal{A}t}+\int_0^te^{\mathcal{A}(t-s)} B(\Theta_s\omega_\tau)	\varphi(s,\omega_\tau)\, ds
	\end{equation}
	admits an exponential dichotomy with bound $\hat{M}$ and exponent $\tilde{\alpha}$ provided in Theorem
	\ref{th-robustness-uniform-ED-co-cycle}.
\end{theorem}

\begin{proof}
	These assumptions on $\mathcal{A}$ implies the existence of an 
	exponential dichotomy for $\{e^{\mathcal{A}t}: t\geq 0\}$, see for instance \cite{Henry-1}. 
	In fact,
	if $\gamma$ is a smooth closed simple curve in 
	$\rho(\mathcal{A})\cap\{\lambda\in \mathbb{C}:\hbox{ Re}\lambda>0\}$ oriented counterclockwise and enclosing $\sigma^+$ let
	\begin{equation*}
	Q=Q(\sigma^+)=\frac{1}{2\pi i}\int_{\gamma}(\lambda -\mathcal{A})^{-1}d\lambda,
	\end{equation*}
	and define $X^+=QX$, $X^-=(I-Q)X$, and $\mathcal{A}^\pm:=\mathcal{A}_\pm$.
	At this scenario, $\mathcal{A}^-$ generates a strongly continuous semigroup on $X^-$,
	$\mathcal{A}^+\in \mathcal{L}(X^+)$, and there are $M\geq 1$, $\beta>0$ such that
	\begin{eqnarray*}
		&\|e^{\mathcal{A}^+t}\|_{\mathcal{L}(X^+)}\leq Me^{\beta t}, \ t\leq 0;\\
		&\|e^{\mathcal{A}^-t}\|_{\mathcal{L}(X^-)}\leq Me^{-\beta t}, \ t\geq 0.
	\end{eqnarray*}
	Now apply Theorem \ref{th-robustness-nonautonomous-random-perturbation-of-nonautonomous}.
	\end{proof}

\begin{remark}
	Similar as in Henry \cite{Henry-1}, Theorem \ref{th-robustness-nonautonomous-random-perturbation-of-autonomous} can be proved in the parabolic case, when  $-\mathcal{A}$ is a sectorial operator with $\mathcal{A}\in \mathcal{L}(X^\alpha, X)$, for a
	unbounded perturbation $B:\mathbb{R}\times\Omega\to \mathcal{L}(X^\alpha, X)$, where $X^\alpha$ is a fractional power of $X$, with $0\leq\alpha<1$. In this situation 
	there exists a $\Theta$-invariant $\epsilon$ such that if $B$ satisfies 
	\begin{eqnarray*}
		&\|B(\Theta_t\omega_\tau)x\|_X\leq b(\omega_\tau) \|x\|_{X^\alpha},\\
		&\sup_{0\leq t\leq  1}\|\int_0^tB(\Theta_s\omega_\tau)x\,ds\|_X<
		q(\omega_\tau)\|x\|_{X^\alpha},
	\end{eqnarray*}
	and $q(\omega_\tau)^\delta b(\omega_\tau)^{1-\delta}\leq \epsilon(\omega_\tau)$ with
	$0<\delta<(1-\alpha)/2$, then any nonautonomous random dynamical system satisfying
	\eqref{eq-existence-nonautonomous-random-perturbation}
	admits an exponential dichotomy in $X^\alpha$.
	
\end{remark}

\begin{remark}
	Note that, Theorem 
	\ref{th-robustness-nonautonomous-random-perturbation-of-autonomous} can be proved independently of Theorem \ref{th-robustness-nonautonomous-random-perturbation-of-nonautonomous} and any nonautonomous random dynamical systems $(\varphi_\epsilon,\Theta)$ satisfying
	\eqref{eq-existence-nonautonomous-random-perturbation} will satisfy the hypotheses of Theorem \ref{th-robustness-nonautonomous-random-perturbation-of-nonautonomous}. 
\end{remark}
%
%
\subsection{Existence and continuity of random hyperbolic solutions}
\label{subsection-randomperturbation-of-autonomous-semilinear-problem}
\quad
\par Now we study a semilinear problem under a nonautonomous random perturbation. We provide
conditions to obtain existence of a bounded random hyperbolic solution for a co-cycle.  
We consider the semilinear problem on a Banach space $X$
\begin{equation}\label{eq-autonomous-semilinear-ODE}
\dot{y}=\mathcal{B}y+f_0(y), \ \ y(0)=y_0
\end{equation}
and a nonautonomous random perturbation of it
\begin{equation}\label{eq-nonautonomous-semilinear-randomODE}
\dot{y}=\mathcal{B}y+f_\eta(\Theta_t\omega_\tau,y), \ \ y(0)=y_0,
\end{equation}
where $(\theta,\Omega)$ is a random flow, and $\{\Theta_t: t\in \mathbb{R}\}$ is a driving flow given by 
$\Theta_t(\omega_\tau)=(t+\tau,\theta_t\omega)$ for every $\omega_\tau\in \mathbb{R}\times \Omega$, and $\eta\in (0,1]$ is a real parameter.
\par We assume that problem \eqref{eq-autonomous-semilinear-ODE} generates a (nonlinear) semigroup $\{T(t): t\geq 0 \}$, and that \eqref{eq-nonautonomous-semilinear-randomODE} generates a (nonlinear) nonautonomous random dynamical system $(\psi_\eta,\Theta)$, for each $\eta\in [0,1]$. 
\par Our goal is to provide conditions on the limit ``$f_\eta\to f_0$'', as $\eta\to 0$, such that, if $\{T(t): t\geq 0 \}$ has a hyperbolic equilibrium $y_0^*$, then there exists a (unique) \textit{random hyperbolic equilibrium} for $(\psi_\eta,\Theta)$ near $y_0^*$, for $\eta>0$ ``small enough''. We first prove existence and continuity of global solutions for \eqref{eq-nonautonomous-semilinear-randomODE} and then show that these solutions exhibit a hyperbolic behavior. 
\begin{definition}
	Let $(\psi,\Theta)$ be a nonautonomous random dynamical system. We say that a map $\zeta:\mathbb{R}\times\Omega\to X$ is a global solution for $(\psi,\Theta)$
	if
	\begin{equation*}
	\psi(t,\omega_\tau)\zeta(\omega_\tau)=\zeta(\Theta_t\omega_\tau), \hbox{ for every }
	t\geq 0.
	\end{equation*}
\end{definition}
\begin{remark}
	For random dynamical systems some authors called these global solutions of \textit{equilibria}. In our work, we will only refer as equilibrium to these kind of global solutions when it exhibits a hyperbolic behavior, in this situation we will call
	\textit{hyperbolic equilibria} or \textit{hyperbolic solutions}.
\end{remark}
\begin{remark}
	Let $(\psi,\Theta)$ be a nonautonomous random dynamical system and a global solution $\zeta$. Then, for each $\omega_\tau$ fixed, the mapping 
	$\mathbb{R}\ni t\mapsto \xi(t,\omega_\tau):=\zeta(\Theta_t\omega_\tau)$
	defines a global solution for the evolution process 
	\begin{equation*}
	\{\psi(t-s,\Theta_s\omega_\tau):t\geq s \}.
	\end{equation*}
\end{remark} 
\par Suppose that $y_0^*\in X$ is a \textbf{hyperbolic equilibrium} for 
\eqref{eq-autonomous-semilinear-ODE}, i.e., the linear operator
$\mathcal{A}:=\mathcal{B}+f^\prime_0(y_0^*)$ generates a $C_0$-semigroup 
$\{e^t:t\geq 0\}$ that admits an exponential dichotomy, see for instance Henry \cite{Henry-1}. 
Let $U$ be an open neighborhood of $y_0^*$ in $X$ such that,
$f_\eta(\omega_\tau,\cdot)\in C^1(U,X)$, for all $\eta\in [0,1]$ and $\omega_\tau\in \mathbb{R}\times \Omega$.
\par Define
\begin{equation*}
\lambda(\eta,\omega_\tau):=\sup_{(t,x)\in\mathbb{R}\times U}\Big\{\|f_\eta(\Theta_t(\omega_\tau),x)-f_0(x)\|_X+
\|(f_\eta)_x(\Theta_t(\omega_\tau),x)-f_0^\prime(x)\|_{\mathcal{L}(X)}\Big\},
\end{equation*}
and we assume that
\begin{equation}\label{eq-fundamental-hypothesis-sthocastic-perturbation}
\lim_{\eta\to 0}\sup_{t\in \mathbb{R}}\lambda(\eta,\Theta_t\omega_\tau)=0.
\end{equation}
Also suppose
\begin{equation}\label{eq-fundamental-hypothesis-contisouly-diferentiability}
\rho(\epsilon):=\sup_{x\in\times U}\sup_{\|h\|\leq \epsilon}
\bigg\{\frac{\|f_0(x+h)-f_0(x)-f_0^\prime(x)h\|_X}{\|h\|_X}\bigg\}\to 0,
\hbox{ as } \epsilon\to 0.
\end{equation}
\begin{theorem}\label{th-existence-random-hyperbolic-solutions}
	Let $y_0^*$ be a hyperbolic equilibrium for \eqref{eq-autonomous-semilinear-ODE} and assume that \eqref{eq-fundamental-hypothesis-sthocastic-perturbation}
	and 
	\eqref{eq-fundamental-hypothesis-contisouly-diferentiability} hold. 
	Given $\epsilon>0$ small enough, there exists a $\Theta$-invariant map
	$\eta_\epsilon:\mathbb{R}\times\Omega\to (0,1]$ such that
	there exists 
	\begin{equation*}
	\mathbb{R}\ni t\mapsto \xi_\eta^*(t,\omega_\tau)\in X, \hbox{ for every } \eta\in (0,\eta_\epsilon(\omega_\tau)],
	\end{equation*}
	a global solution of $\{\psi_\eta(t-s,\Theta_s\omega_\tau):t\geq s \}$ 
	such that
	\begin{equation*}
	\sup_{t\in \mathbb{R}} \|\xi_{\eta}^*(t,\omega_\tau)-y_0^*\|_X<\epsilon, \hbox{ for every } \eta\in (0,\eta_\epsilon(\omega_\tau)].
	\end{equation*}
\end{theorem}

\begin{proof}
	Let $y$ be a global solution of  
	\eqref{eq-nonautonomous-semilinear-randomODE}. 
	Then, if we define
	$\phi=y-y_0^*$, it satisfies 
	\begin{equation*}
	\dot{\phi}=\mathcal{A}\phi+g_\eta(\Theta_t\omega_\tau,\phi),
	\end{equation*}
	where $g_\eta(\Theta_t\omega_\tau,\phi)=f_\eta(\Theta_t\omega_\tau,y_0^*+\phi)-f_0(y_0^*)-
	f_0^\prime(y_0^*)\phi$, so that
	\begin{equation}\label{eq-existence_hyperbolic-solution-FVC}
	\phi(t)=e^{\mathcal{A}(t-\tau)}\phi(\tau)+
	\int_{\tau}^{t}e^{\mathcal{A}(t-s)}g_\eta(\Theta_s\omega_\tau,\phi(s))ds
	\end{equation}
	Hence, if we project $Q$ and $I-Q$ and take limits we obtain
	\begin{equation*}
	\phi(t)=
	\int_{-\infty}^{+\infty}G_\mathcal{A}(t,s)g_\eta(\Theta_s\omega_\tau,\phi(s))ds,
	\end{equation*}
	where $G$ is the Green function associated with the semigroup $\{e^{\mathcal{A}t}:t\geq 0 \}$
	and projection $Q$.
	\par Consequently, a complete bounded solution to 
	\eqref{eq-existence_hyperbolic-solution-FVC} exists in a small neighborhood of
	$x=0$, if and only if, the operator
	\begin{equation*}
	\mathcal{I}_{\omega_\tau,\eta}(\phi)(t)=
	\int_{-\infty}^{+\infty}G_\mathcal{A}(t,s)g_\eta(\Theta_s\omega_\tau,\phi(s))ds
	\end{equation*}
	has a unique fixed point in the set
	\begin{equation*}
	\mathfrak{X}_\epsilon:=\Big\{\phi:\mathbb{R}\rightarrow X:\sup_{t\in \mathbb{R}} \|\phi(t)\|_X\leq \epsilon \Big\}
	\end{equation*}
	for a given $\epsilon>0$ small. This follows by a fixed point argument for
	$\mathcal{I}_{\omega_\tau,\eta}$ for each $\omega\in \Omega$ fixed.
	\par Indeed, 
	let $\epsilon_1>0$ such that
	\begin{equation}
	\|f_0^\prime(y_0^*+h)-f_0^\prime(y_0^*)\|_{\mathcal{L}(X)}<\frac{1}{6 M\beta^{-1}},
	\hbox{ for every } \|h\|_X<\epsilon_1,
	\end{equation}
	and $\epsilon_2\in (0,1/2)$ be such that $\rho_0(\epsilon)<1/6M\beta^{-1}$, for every $0<\epsilon<\epsilon_2$, where $M>1$ is the bound, and $\beta>0$ is the exponent of the exponential dichotomy of $\{e^{\mathcal{A}t}: t\geq 0\}$.
	Define $\epsilon_0=\min{\epsilon_1,\epsilon_2/2}$ and 
	for a given $\omega_\tau\in \mathbb{R}\times \Omega$ fixed
	and 
	$\epsilon\in (0,\epsilon_0)$, 
	define $\eta_\epsilon(\omega_\tau)>0$ such that
	\begin{equation*}
	\lambda(\eta,\omega_\tau)<\frac{\epsilon}{6M\beta^{-1}}, \hbox{ for every } \eta\in (0,\eta_\epsilon(\omega_\tau)].
	\end{equation*}
	Then, it is possible to prove that $\mathcal{I}_{\omega_\tau,\eta}$ maps 
	$\mathfrak{X}_\epsilon$ into itself. 
	In fact, for $\phi \in \mathfrak{X}_\epsilon$
	\begin{eqnarray*}
		\|g_\eta(\Theta_t\omega_\tau,\phi(t))\|_X&\leq& \|f_\eta(\Theta_t\omega_\tau,y_0^*+\phi(t))-f_0(y_0^*+\phi(t))\|_X+
		\rho(\epsilon)\epsilon\\
		&\leq& \lambda(\eta,\omega_\tau)+\rho_0(\epsilon)\epsilon,
	\end{eqnarray*}
	hence
	\begin{equation*}
	\|\mathcal{I}_{\eta,\omega}\phi(t)\|_X\leq 2\beta^{-1}M \|g_\eta(\Theta_t\omega_\tau,\phi(t))\|_X<\epsilon,
	\end{equation*}
	and that $\mathcal{I}_{\omega_\tau,\eta}$ is a contraction.
	Let $\phi_1,\phi_2 \in \mathfrak{X}_\epsilon$
	\begin{eqnarray*}
		& &\|g_\eta(\Theta_t\omega_\tau, \phi_1(t))- g_\eta(\Theta_t\omega_\tau, \phi_2(t))\|_X\\
		& &\leq \|f_\eta(\Theta_t\omega_\tau, y_0^*+\phi_1(t))-f_\eta(\Theta_t\omega_\tau, y_0^*+\phi_2(t))-f_0^\prime(y_0^*)(\phi_1(t)-\phi_2(t))\|_X\\
		& &\leq \bigg[\lambda(\eta,\omega_\tau)+\rho(\epsilon)+\|f_0^\prime(y_0^*+\phi_1)-f_0^\prime(y_0^*)\|_{\mathcal{L}(X)}\bigg]\,\|(\phi_1(t)-\phi_2(t))\|_X.
	\end{eqnarray*}
	Then
	\begin{equation*}
	\|\mathcal{I}_{\eta,\omega}\phi_1(t)-\mathcal{I}_{\eta,\omega}\phi_1(t)\|_X\leq \frac{1}{2}\,\|\phi_1(t)-\phi_2(t)\|_X.
	\end{equation*}
	Therefore, there exists a fixed point $\phi_\eta^*(\cdot,\omega_\tau)$ of $\mathcal{I}_{\omega_{\tau},\eta}$ in 
	$\mathfrak{X}_\epsilon$, and the global solution of \eqref{eq-nonautonomous-semilinear-randomODE} is given by
	$\xi_\eta^*(\cdot,\omega_{\tau})=\phi_\eta^*(\cdot,\omega_\tau)+y_0^*$.
\end{proof}
\par  As in the deterministic case, see Carvalho and Langa, \cite{Carvalho-Langa-2}, these solutions 
$\{\xi_\eta^*\}$ play the role of an 
\textit{hyperbolic equilibrium} for \eqref{eq-nonautonomous-semilinear-randomODE}.  
Given $\epsilon>0$ define, for each $\omega_\tau$ fixed and $\eta\in (0,\eta_\epsilon(\omega_\tau)]$,
\begin{equation*}
\zeta^*_\eta(\tau,\omega):=\xi^*_\eta(0,\omega_\tau).
\end{equation*}
Note that, for each $\omega_\tau$ fixed, there exists $\eta_\epsilon(\omega_\tau)>0$ such that the mapping $\mathbb{R}\ni t\mapsto \xi_\eta^*(t,\omega_\tau):=\zeta^*_\eta(\Theta_t\omega_\tau)$, $t\in \mathbb{R}$ is a complete solution for
\begin{equation}\label{eq-nonautonomous-semilinear-randomODE-translated}
\dot{x}=\mathcal{B}x+f_\eta(\Theta_t\omega_\tau,x), \ \eta\in(0,\eta_\epsilon(\omega_\tau)].
\end{equation}
Then, to ensure that $\xi^*_\eta$ exhibits a hyperbolic behavior, we linearized problem \eqref{eq-nonautonomous-semilinear-randomODE-translated} over $\zeta^*_\eta$ and guarantee that the associated linear nonautonomous random dynamical system admits an exponential dichotomy. 

\begin{remark}\label{th-existence-random-hyperbolic-solutions-proof}
	Let $\omega_\tau\in \mathbb{R}\times \Omega$ be fixed, $x_\eta(\cdot,\omega_\tau)$ a solution of \eqref{eq-nonautonomous-semilinear-randomODE} and define  $z_\eta(t)=x_\eta(t,\omega)-\zeta_\eta^*(\Theta_t\omega_\tau)$, for each $t\geq 0$ and 
	$\eta\in (0,\eta_\epsilon(\omega_{\tau})]$. Then
	\begin{equation}\label{eq-semilinear-linearization-co-cycle}
	\dot{z}=\mathcal{A}z+B_\eta(\Theta_t\omega_\tau)z+h_\eta(\Theta_t\omega_\tau,z),
	\end{equation}
	where $B_\eta(\Theta_t\omega_\tau)=
	(f_\eta)_z(\Theta_t\omega_\tau,\zeta_\eta^*(\Theta_t\omega_\tau))-f_0^\prime(y_0^*)$,
	and \begin{equation*}
	h_\eta(\Theta_t\omega_\tau,z):=f_\eta(\Theta_t\omega_\tau,\zeta_\eta^*(\Theta_t\omega_\tau)+z)-f_\eta(\Theta_t\omega_\tau,\zeta_\eta^*(\Theta_t\omega_\tau))- (f_\eta)_z(\Theta_t\omega_\tau,\zeta_\eta^*(\Theta_t\omega_\tau))z.
	\end{equation*} 
	Thus, $0$ is a globally defined bounded solution for \eqref{eq-semilinear-linearization-co-cycle} and
	$h_\eta(\Theta_t\omega_\tau,0)=0$, $(h_\eta)_z(\Theta_t\omega_\tau,0)=0\in \mathcal{L}(X)$.
	\par We consider the linearization of problem \eqref{eq-nonautonomous-semilinear-randomODE} about $\xi^*_\eta$
	\begin{eqnarray}\label{eq-semilinear-linearization-co-cycle-2}
	\dot{z}&=&\mathcal{B}+(f_\eta)_z(\Theta_t\omega_\tau,\zeta_\eta^*(\Theta_t\omega_\tau))z\\
	&=&\mathcal{A}z+B_\eta(\Theta_t\omega_\tau)z.
	\end{eqnarray}
	Note that for each $\omega_\tau\in \mathbb{R}\times \Omega$ fixed
	\begin{equation*}
	\lim_{\eta\to 0}\sup_{t\in\mathbb{R}}\|B_\eta(\Theta_t\omega_\tau)\|_{\mathcal{L}(X)}=0.
	\end{equation*}
	Then we choose $\tilde{\eta}_\epsilon\leq \eta_\epsilon$ $\Theta$-invariant and consider a linear co-cycle 
	$\varphi_\epsilon(t,\omega_{\tau})=\varphi_{\tilde{\eta}_\epsilon(\omega_\tau)}(t,\omega_\tau)$ satisfying
	\begin{equation*}
	\varphi_\epsilon(t,\tau,\omega)x_0=e^{\mathcal{A}t}x_0+
	\int_{0}^te^{\mathcal{A}(t-s)}B_{\tilde{\eta}_\epsilon(\omega_{\tau})}(\Theta_s\omega_\tau)
	\varphi_\epsilon(s,\tau,\omega)x_0ds.
	\end{equation*}	
	Since $\{e^{\mathcal{A}t}: t\geq 0\}$ admits an exponential dichotomy, we apply Theorem \ref{th-robustness-nonautonomous-random-perturbation-of-nonautonomous} to guarantee that the linear co-cycle
	$(\varphi_\epsilon,\Theta)$ admits an exponential dichotomy for each suitable small $\epsilon>0$. Thus $\zeta^*_\epsilon:=\zeta^*_{\tilde{\eta}_\epsilon}$ is a global solution that exhibits a \textbf{hyperbolic behavior}.
\end{remark}
%
%
\par  The discussion above suggests a proper definition of \textit{random hyperbolic solution} for
a nonautonomous random differential equation.
\begin{definition}
	Let $B$ be a generator of a strongly continuous semigroup, and
	$f:\mathbb{R}\times\Omega\times X\to X$ such that
	for each $(t,\omega)$ fixed 
	$X\ni x\mapsto f(t,\omega,x) $ is differentiable, and $\zeta:\mathbb{R}\times \Omega\to X$ be a bounded global solution of
	\begin{equation}\label{eq-definition-random-hyperbolic-solution-correspondent}
	\dot{x}=Bx+f(\Theta_t\omega_\tau, x), \ t\geq 0, x(0)=x_0\in X.
	\end{equation}
	We say that $\zeta$ is a \textbf{random hyperbolic solution} of 
	\eqref{eq-definition-random-hyperbolic-solution-correspondent}
	if there exists a linear nonautonomous random dynamical system
	$(\varphi,\Theta)$ satisfying
	\begin{equation*}
	\varphi(t,\omega_\tau)=e^{Bt}+
	\int_{0}^te^{B(t-s)}D_xf(\Theta_s\omega_\tau,\zeta(\Theta_s\omega_\tau))
	\varphi(s,\omega_\tau)ds, \hbox{ for all } \omega_\tau \in 
	\mathbb{R}\times \Omega,
	\end{equation*}
	and $(\varphi,\Theta)$ admits an exponential dichotomy.
\end{definition}

\par Joining the results of Theorem \ref{th-existence-random-hyperbolic-solutions} and
Remark \ref{th-existence-random-hyperbolic-solutions-proof} we obtain:
\begin{theorem}[Existence and continuity of random hyperbolic solutions]\label{th-existence-random-hyperbolic-solutions-complete}
	Let $y_0^*$ be a hyperbolic equilibrium for \eqref{eq-autonomous-semilinear-ODE} and assume that \eqref{eq-fundamental-hypothesis-sthocastic-perturbation} and 
	\eqref{eq-fundamental-hypothesis-contisouly-diferentiability} hold. 
	Given $\epsilon>0$ small enough, there exists a $\Theta$-invariant map
	$\eta_\epsilon:\mathbb{R}\times\Omega\to (0,1]$ and a global solution $\zeta_{\epsilon}^*:\mathbb{R}\times \Omega\to X$ of 
	$(\psi_{\eta_\epsilon},\Theta)$ defined by
	\begin{equation*}
	\zeta_{\epsilon}^*(\omega_{\tau}):=\zeta_{\eta_\epsilon(\omega_\tau)}^*(\omega_\tau)	\end{equation*} 
	such that $\zeta_{\epsilon}^*$ is a random hyperbolic solution of $(\psi_{\eta_\epsilon},\Theta)$, with 
	\begin{equation*}
	\sup_{t\in \mathbb{R}}\|\zeta_{\epsilon}^*(\Theta_t\omega_\tau)-y_0^*\|_X<\epsilon.
	\end{equation*}
\end{theorem}

\begin{remark}
	Theorem \ref{th-existence-random-hyperbolic-solutions-complete} provides existence and continuity of \textit{random hyperbolic solutions} from nonautonomous random perturbations of a hyperbolic problem. However, this result of persistence can be proved in a general context. In other words, following similar steps, it is possible to prove that \textit{random hyperbolic solutions} are stable under (random nonautonomous) perturbations.
\end{remark}

\begin{remark}
	In the parabolic case, when $-\mathcal{A}$ is sectorial, with $\mathcal{A}\in \mathcal{L}(X^\delta,X)$, $0<\delta<1$, where $X^\delta$ is a fractional power of $X$, we cannot assume that the nonlinearity $f_0:U\subset X\to X$ is differentiable, see \cite{Bortolan-Cardoso-etal}. 
	We have to assume that 
	the hyperbolic equilibrium $y_0^*$ is in $X^\delta$ and that $U$ is a open neighborhood of 
	$y_0^*$ in $X^\delta$ such that 
	$f_0:U\subset X^\delta\to X$ is differentiable with derivative 
	$f^\prime(y_0^*)\in \mathcal{L}(X^\delta,X)$.
	Also, we have to use a slightly different estimative on the Green function of $\{e^{\mathcal{A}t}:t\geq 0\}$
	\begin{eqnarray*}
		\|G_\mathcal{A}(t,s)\|_{\mathcal{L}(X,X^\delta)}&\leq& D(M,\delta) (t-s)^{-\delta} e^{-\beta|t-s|},  0<t-s\leq 1\\
		\|G_\mathcal{A}(t,s)\|_{\mathcal{L}(X,X^\delta)}&\leq& D(M,\delta) e^{-\beta|t-s|}, \hbox{ otherwise},
	\end{eqnarray*}
	where $D=D(M,\delta)$ is a constant, see \cite{Henry-1}. Under these conditions the proof of existence and continuity of bounded random hyperbolic equilibrium on $X^\delta$ is analogous to the argument used in 
	Theorem \ref{th-existence-random-hyperbolic-solutions-complete}. 
\end{remark}

\subsection{Stochastic perturbations of autonomous problems}\label{subsec-app-stochastic}
\quad
\par In this subsection we apply the results established in this work 
to study stability under a nonautonomous stochastic perturbation of a hyperbolic autonomous problem. We consider the following family of \textit{Stratonovich stochastic 
	differential equations} with a nonautonomous multiplicative 
white noise
\begin{equation}\label{eq-applications-stochastic-perturbation}
dy=\mathcal{B}ydt+f(y)dt +\eta\kappa_ty\circ dW_t, \ t\geq \tau, \ y(\tau)=y_\tau,
\end{equation}
where $\mathcal{B}$ is a generator of a strongly continuous semigroup 
$\{e^{\mathcal{B}t}: t\geq 0 \}$ on $X$, the family
$\{W_t:t\in \mathbb{R}\}$ is the standard Wiener process, see \cite{Arnold}, and
$\kappa:\mathbb{R}\rightarrow \mathbb{R}^+$ is continuously differentiable,
$\eta>0$.
\par We apply the results of Subsection \ref{subsection-randomperturbation-of-autonomous-semilinear-problem}, by a formal change of variable we show how to modify problem \eqref{eq-applications-stochastic-perturbation} to a nonautonomous random differential equations. We prove existence and continuity of \textit{random hyperbolic solutions} for the associated nonautonomous differential equation.
\par The canonical sample space of a Wiener process is
$\Omega:=C_0(\mathbb{R})$ the set of continuous functions over $\mathbb{R}$ which are 
$0$ at $0$ equipped with the compact open topology. We denote $\mathcal{F}$ the associated Borel $\sigma$-algebra. Let $\mathbb{P}$ be the Wiener probability measure on 
$\mathcal{F}$ which is given by the distribution of a two-sided Wiener process with trajectories in $C_0(\mathbb{R})$. The flow $\theta$ is given by the Wiener shifts
\begin{equation*}
\theta_t\omega(\cdot)=\omega(t+\cdot)-\omega(t), \ t\in \mathbb{R}, \ \omega\in \Omega.
\end{equation*}
\par In order to obtain a nonautonomous random differential equation from \eqref{eq-applications-stochastic-perturbation} we consider an auxiliary scalar stochastic differential equation 
\begin{equation}\label{eq-orstein-uhlenbeck-sde}
dz_t+zdt=dW_t.
\end{equation}
This problem has a stationary solution known as the Ornstein-Uhlenbeck process.
\begin{lemma}
	There exists a $\theta$-invariant subset 
	$\tilde{\Omega}\in \mathcal{F}$ of full measure such that
	$\lim_{t\to \pm\infty}\frac{|\omega(t)|}{t}=0, \ \omega\in \tilde{\Omega}$
	and, for such $\omega$, the random variable given by
	\begin{equation*}
	z^*(\omega)=-\int_{-\infty}^0 e^{ s}\omega(s)ds
	\end{equation*}
	is well defined. Moreover,
	for $\omega\in \tilde{\Omega}$, the mapping
	$(t,\omega)\mapsto z^*(\theta_t\omega)$
	is a stationary solution of \eqref{eq-orstein-uhlenbeck-sde}
	with continuous trajectories, and 
	\begin{equation*}\label{eq-sublinear-growth-OU}
	\lim_{t\to \pm\infty}\frac{|z^*(\theta_t\omega)|}{t}=0, \ \forall \,\omega\in \tilde{\Omega}.
	\end{equation*}
\end{lemma}
The proof of the following Lemma can be found on \cite{Caraballo-Kloeden-Schmalfu}.
\par Let $y$ be a solution for 
\eqref{eq-applications-stochastic-perturbation} and consider 
$v(t,\omega):=e^{-\eta\kappa_tz^*(\theta_t\omega)}y(t,\omega)$. Hence, $v$ has to satisfy the following nonautonomous random differential equation  
\begin{equation*}
\dot{v} =\mathcal{B}v+e^{-\eta\kappa_tz^*(\theta_t\omega)}
f(e^{\eta\kappa_tz^*(\theta_t\omega)}v)
+\eta [\kappa_t-\dot{\kappa}_t] z^*(\theta_t\omega)v,
\end{equation*}
Define $f_\eta(t,\omega,v):=e^{-\eta\kappa_tz^*(\omega)}
f(e^{\eta\kappa_tz^*(\omega)}v)$, and
$B_\eta(t,\omega)v:=\eta [\kappa_t-\dot{\kappa}_t] z^*(\omega)v$.
\par In order to obtain hyperbolic solutions for \eqref{eq-applications-stochastic-perturbation} we study existence of random hyperbolic solutions for
\begin{equation}\label{eq-application-stochastic-nonautonomous-random}
\dot{v}= \mathcal{B}v+f_\eta(\Theta_t\omega_\tau,v)+B_\eta(\Theta_t\omega_\tau) v,
\ t\geq 0, v(0)=v_0\in X.
\end{equation}
\par To apply our results, we have to check hypotheses of Theorem
\ref{th-existence-random-hyperbolic-solutions-complete} to guarantee existence of random hyperbolic solutions \eqref{eq-application-stochastic-nonautonomous-random}.
\par To that end, we choose any differentiable positive real function $\kappa$ for which 
there are random variables $m_1,m_2>0$ such that
\begin{equation*}
m_1(\omega):=\sup_{t\in \mathbb{R}}\{|\kappa_tz^*(\theta_t\omega)|\}<\infty, \hbox{ and }
m_2(\omega):=
\sup_{t\in \mathbb{R}}\{|[\kappa_t-\dot{\kappa}_t]z^*(\theta_t\omega)|\}<\infty.
\end{equation*}
\par Define, for each $\omega_\tau\in \mathbb{R}\times \tilde{\Omega}$,
$c(\tau,\omega):=(\kappa_\tau-
\dot{\kappa}_\tau)z^*(\omega)$,
\begin{equation*}
M(\tau,\omega):= \sup_{t\in \mathbb{R}}\{|c(\Theta_t(\omega_\tau))|\}\leq 
m_2(\theta_{-\tau}\omega).
\end{equation*}
Note that $\sup_{t\in \mathbb{R}}M(\Theta_t\omega_\tau)=M(\omega_\tau).$
Then $
\sup_{t\in\mathbb{R}}\|B_\eta(\Theta_t\omega_\tau)\|_{\mathcal{L}(X)}=\eta M(\omega_\tau).
$
\par Let $y_0^*$ be a hyperbolic equilibrium of $\dot{y}=\mathcal{B}y+f(y)$, and $U$ be a neighborhood of $y_0^*$ such that $f_\eta(\omega_\tau,\cdot)\in C^1(U,X)$, for every $\eta\in [0,1]$ and $\omega_\tau\in \mathbb{R}\times \tilde{\Omega}$. 
We claim that for each $\omega_\tau\in \mathbb{R}\times\tilde{\Omega}$
\begin{equation}
\sup_{(t,x)\in \mathbb{R}\times U}\Big\{\|f_\eta(\Theta_t\omega_\tau,x)-f(x)\|_X+
\|D_xf_\eta(\Theta_t\omega_\tau,x)-f^\prime(x)+B_\eta(\Theta_t\omega_\tau)\|_{\mathcal{L}(X)}\Big\}\to 0,
\end{equation}
as $\eta\to 0$. The derivative of $f_\eta(\Theta_t\omega_\tau,\cdot)$ on $x$ is
\begin{equation*}
D_xf_\eta(\Theta_t\omega_\tau,x)=
f^\prime(e^{\eta\kappa_{t+\tau}z^*(\theta_t\omega)}x).
\end{equation*}
Since, for every $(\tau,\omega)\in \mathbb{R}\times\tilde{\Omega}$ the limit $\eta\sup_{t\in \mathbb{R}}|\kappa_{t+\tau}z^*(\theta_{t}\omega)|$ goes to $0$, as $\eta\to 0$, by continuity 
$
\sup_{t\in \mathbb{R}}|e^{\eta\kappa_{t+\tau}z^*(\theta_t\omega)}-1|\to 0, \hbox{ as } \eta\to 0.$
Similarly, for $f^\prime$ we have that
\begin{equation*}
\sup_{t\in \mathbb{R}}
\|f^\prime(e^{\eta\kappa_{t+\tau}z^*(\theta_t\omega)}x)-f^\prime(x)\|_{\mathcal{L}(X)}\to 0, 
\hbox{ as } \eta \to 0,
\end{equation*}
\begin{eqnarray*}
	\|f_\eta(\Theta_{t}\omega_\tau,x)-f(x)\|_X&\leq& 
	e^{-\eta\kappa_{t+\tau}z^*(\theta_t\omega)}\|f(e^{\eta\kappa_{t+\tau}z^*(\theta_t\omega)}x)-f(x)\|_X\\
	&+&
	|e^{-\eta\kappa_{t+\tau}z^*(\theta_t\omega)}-1|\|f(x)\|_X,
\end{eqnarray*}
and both terms on the right side go to zero when $\eta\to 0$, uniformly on $t$.
Therefore, for each $\omega_\tau\in \mathbb{R}\times \tilde{\Omega}$ we have that
$\sup_{s\in \mathbb{R}}\lambda(\eta,\Theta_s\omega_{\tau})=\lambda(\eta,\omega_\tau)\to 0, \hbox{ as } \eta\to 0.$
Hence, we can apply Theorem \ref{th-existence-random-hyperbolic-solutions-complete} to 
\eqref{eq-application-stochastic-nonautonomous-random} to prove existence and continuity of hyperbolic solutions.
\begin{theorem}\label{th-existence-random-hyperbolic-solutions-complete-stochastic}
	Let $y_0^*$ be a hyperbolic equilibrium for $\dot{y}=\mathcal{B}y+f(y)$. 
	Given $\epsilon>0$ small enough, there exists a $\Theta$-invariant map
	$\eta_\epsilon:\mathbb{R}\times\Omega\to (0,1]$ and a global solution $\zeta_{\epsilon}^*:\mathbb{R}\times \Omega\to X$ of 
	$(\psi_{\eta_\epsilon},\Theta)$ defined by
	\begin{equation*}
	\zeta_{\epsilon}^*(\omega_{\tau}):=\zeta_{\eta_\epsilon(\omega_\tau)}^*(\omega_\tau)	\end{equation*} 
	such that $\zeta_{\epsilon}^*$ is a random hyperbolic solution of $(\psi_{\eta_\epsilon},\Theta)$, with 
	\begin{equation*}
	\sup_{t\in \mathbb{R}}\|\zeta_{\epsilon}^*(\Theta_t\omega_\tau)-y_0^*\|_X<\epsilon.
	\end{equation*}
	Moreover, $\zeta_{\epsilon}^*$ defines a stochastic process, i.e., 
	the mapping $\tilde{\Omega}\ni \omega\mapsto \zeta_{\epsilon}^*(t,\omega)\in X$ is $(\tilde{\mathcal{F}},\mathcal{B}_X)$-measurable, for each $t\in \mathbb{R}$.
\end{theorem}

\begin{proof}
	Given $\epsilon>0$ small enough, thanks to Theorem \ref{th-existence-random-hyperbolic-solutions-complete}, we know that for every $\omega_\tau\in \mathbb{R}\times\tilde{\Omega}$ there exists a $\Theta$-invariant map $\eta_\epsilon:\mathbb{R}\times \tilde{\Omega}\to (0,1]$ such that 
	$\zeta_{\epsilon}^*:\mathbb{R}\times\tilde{\Omega}\to X$ is a random hyperbolic solution for
	\eqref{eq-application-stochastic-nonautonomous-random}.
	\par Let us prove that $\zeta_{\epsilon}^*(t,\cdot)$ is $(\tilde{\mathcal{F}},\mathcal{B}_X)$-measurable. 
	Consider $\phi_{\eta_\epsilon(\omega_\tau)}\in \mathfrak{X}_\epsilon$
	the fixed point of operator $\mathcal{I}_{\omega_\tau,\eta_{\epsilon}(\omega_\tau)}$, see the proof of Theorem \ref{th-existence-random-hyperbolic-solutions}.
	\par We claim that the mapping
	$\tilde{\Omega}\ni\omega\mapsto \phi_{\eta_\epsilon(\omega_\tau)}(0,\omega_\tau)$ is 
	$(\tilde{\mathcal{F}},\mathcal{B}_X)$-measurable, where $\tilde{\mathcal{F}}$ is the $\sigma$-algebra of $\tilde{\Omega}$.
	Indeed, since
	$\phi_{\eta_\epsilon(\omega_\tau)}(0)=\mathcal{I}_{\omega_\tau,\eta_{\epsilon}(\omega_\tau)}(\phi_{\eta_\epsilon(\omega_\tau)})(0),$
	we know that 
	\begin{equation*}
	\phi_{\eta_\epsilon(\omega_\tau)}(0)=\lim_{n\to +\infty}\mathcal{I}_{\omega_\tau,\eta_{\epsilon}(\omega_{\tau})} ^n(\tilde{0})(0)
	\end{equation*}
	where $\tilde{0}\in \mathfrak{X}_\epsilon$.
	Thus, the claim follows from the fact that, for each $\tau \in \mathbb{R}$, the mapping 
	$\omega\mapsto \mathcal{I}_{\omega_\tau,\eta_{\epsilon}(\omega_{\tau})}(\tilde{0})(0)$ 
	is $(\tilde{\mathcal{F}},\mathcal{B}_X)$-measurable.
	\par Therefore, for each $\tau\in \mathbb{R}$ fixed, the mapping $\tilde{\Omega}\ni\omega\mapsto\zeta_{\epsilon}^*(\omega_\tau):=\phi_{\eta_{\epsilon}\omega_\tau}^*(0,\omega_\tau)+y_0^*$ is $(\tilde{\mathcal{F}},\mathcal{B}_X)$-measurable, and the proof is complete.
\end{proof}

\par To end this section we present two special examples: the linear case, and an application in partial differential equations.
\begin{example}
	If we assume that $\mathcal{A}$ generates a  $C_0$-semigroup satisfying conditions of Theorem \ref{th-robustness-nonautonomous-random-perturbation-of-autonomous}, then $y=0$ is a random hyperbolic solution for 
	\begin{equation*}
	dy=\mathcal{A}ydt +\eta\kappa_ty\circ dW_t, \ t\geq \tau,
	\end{equation*}
	for suitable small $\eta>0$.
	In fact, the associated nonautonomous random differential equation is
	\begin{equation*}
	\dot{v} =\mathcal{A}v+
	\eta[\kappa_t-\dot{\kappa}_t] z^*(\theta_t\omega)v, \ t\geq \tau.
	\end{equation*}
	Thus, by Theorem \ref{th-robustness-nonautonomous-random-perturbation-of-autonomous}, for every $\epsilon>0$ suitable small there exists a $\Theta$-invariant 
	$\eta_\epsilon:\mathbb{R}\times\Omega\to (0,1]$ such that
	$A_\epsilon(\omega_\tau)=\mathcal{A}
	+ \eta_\epsilon(\omega_{\tau})[\kappa_\tau-\dot{\kappa}_{\tau}] z^*(\omega)v$ generates a linear nonautonomous random dynamical system that admits an exponential dichotomy.
\end{example}	


\par We now apply Theorem \ref{th-existence-random-hyperbolic-solutions-complete-stochastic}
to a \textit{strongly damped wave equation}.
\begin{example}
	Let $D$ be a bounded smooth domain in $\mathbb{R}^3$,
	$A:D(A)\subset L^2(D)\to L^2(D)$ is $-\Delta$ with Dirichlet boundary condition,
	and $f:\mathbb{R}\to \mathbb{R}$ is twice differentiable.
	Consider the damped wave equation 
	\begin{equation*}
	u_{tt}+\beta u_t-\Delta u=f(u), \hbox{ in } D
	\end{equation*}
	with boundary condition $u=0$, in $\partial D$. The initial data will be taken in
	the space $X=H^1_0(D)\times L^2(D)$. 
	Hence we can consider an abstract evolutionary equation in $X$:
	\begin{equation}\label{example-eq-evolutionary-wave-eq}
	\dot{y}=\mathcal{B}y+ F(y),
	\end{equation}
	where
	\begin{equation*}
	y=\begin{pmatrix}
	y_1\\
	y_2 
	\end{pmatrix}\in X, \ \ 
	\mathcal{B}=\begin{pmatrix}
	0 & I \\
	-A & -\beta 
	\end{pmatrix}, \ \ 
	F(y)=\begin{pmatrix}
	0 \\
	f^e(y_1) 
	\end{pmatrix},
	\end{equation*}
	with $f^e:H_0^1(D)\to L^2(D)$ is given by $f^e(y_1)(x)=f(y_1(x))$ for $x\in D$. 
	Moreover, $f^e$ is continuously differentiable on the second variable, see \cite{Arrieta-Carvalho-Hale-92}.
	The hyperbolic equilibrium points of \eqref{example-eq-evolutionary-wave-eq} are of the form
	$y_0^*=(u_0^*,0)$ where $u_0^*$ is a solution of 
	\begin{equation*}
	-\Delta u=g_0(u),
	\end{equation*} 
	such that $0\notin \sigma(-\Delta+D_xf^e(u_0^*)Id_X)$.
	Now we consider a nonautonomous multiplicative white noise on \eqref{example-eq-evolutionary-wave-eq}, i.e.,
	\begin{equation}\label{example-eq-evolutionary-wave-eq-noise}
	\dot{y}=\mathcal{B}y+ F(y)+\tilde{\eta}\kappa_{t} y\circ dW_t,
	\end{equation}
	where
	\begin{equation*}
	\tilde{\eta}=\begin{pmatrix}
	\eta & 0 \\
	0 & 0
	\end{pmatrix},\ \
	\begin{pmatrix}
	0 & 0 \\
	0 & \eta
	\end{pmatrix}  \hbox{ or }
	\begin{pmatrix}
	\eta & 0 \\
	0 & \eta
	\end{pmatrix},
	\end{equation*}
	and we may choose where to consider the noise, on the position $u$, or velocity $u_t$ or both.
	\par  Let $A$ be a $2\times 2$ - matrix with real coefficients, \textit{the exponential of $A$} is represented by $\textbf{e}^A$.
	Let $y$ be a solution of \eqref{example-eq-evolutionary-wave-eq-noise}, then 
	apply the change of variables
	$v(t,\omega)=\textbf{e}^{-\tilde{\eta}\kappa_tz^*(\theta_t\omega)}y$ to obtain
	\begin{equation}\label{example-eq-evolutionary-wave-eq-random}
	\dot{v}=\mathcal{B}v+ F_\eta(\Theta_t\omega_\tau,v)+\tilde{\eta}(\kappa_{t} -\dot{\kappa_t})z^*(\theta_t\omega)v,
	\end{equation}
	where $F_\eta(\omega_\tau,v):=\textbf{e}^{\tilde{\eta}\kappa_\tau z^*(\omega)}F(\textbf{e}^{-\tilde{\eta}\kappa_\tau z^*(\omega)}v)$. 
	Note that, for each $\eta\in [0,1]$ and $\omega_\tau\in \mathbb{R}\times \Omega$ fixed,
	problem \eqref{example-eq-evolutionary-wave-eq-random}
	is well-posed, see Arrieta \textit{et al.} \cite{Arrieta-Carvalho-Hale-92}.
	\par Now, we apply Theorem
	\ref{th-existence-random-hyperbolic-solutions-complete-stochastic}
	for \eqref{example-eq-evolutionary-wave-eq-random} to ensure that 
	for each hyperbolic equilibrium $y_0^*$ and $\epsilon>0$
	there exists a random hyperbolic solution that defines stochastic processes $\{\xi^*_\epsilon(t,\cdot):\tilde{\Omega}\to X:\,t\in \mathbb{R} \}$ for each $\epsilon> 0$ such that $\sup_{t\in \mathbb{R}}\|\xi^*_\epsilon(\Theta_t\omega_\tau)-y_0^*\|\to 0$ as $\epsilon\to 0$.
\end{example}

\begin{remark}
	We note that it is not clear yet if the dynamical systems generated by \eqref{eq-autonomous-perturbation-by-nonautonomous-random} and 
	\eqref{eq-applications-stochastic-perturbation} are conjugated. This is not a simple task because our change of variables is not stationary because of the presence of $\kappa_t$.
	However, for every solution of
	\eqref{eq-application-stochastic-nonautonomous-random} it is possible to find a corresponding solution for the nonautonomous random differential equations 
	\eqref{eq-autonomous-perturbation-by-nonautonomous-random} and we analyze properties of 
	the later. 
	More precisely, we know now under which conditions it is possible to guarantee that global solutions of \eqref{eq-application-stochastic-nonautonomous-random} exhibit a hyperbolic behavior.
\end{remark}

\begin{remark}
	The study of this subsection could be applied to \textit{It\^o stochastic differential equations} with multiplicative white noise, by modifying to a Stratonovich stochastic differential equation, whenever is possible.
\end{remark}

\section{Conclusions}
\par The results of this paper on the robustness of exponential dichotomies can be extended to nonuniform or tempered, as it is done in \cite{Barreira-Dragicevi-Valls,Zhou-Lu-Zhang-1}. Since our goal was to study the effect of a small bounded noise on autonomous problems, it was expected to obtain nonuniform behavior on our hyperbolicity. We reinforce that to consider a bounded noise is sensibles in real life applications \cite{Caraballo-Colucci-Cruz,Caraballo-Garrido-Cruz,Carr}, and it was a crucial to prove permanence of the hyperbolicity. 
This is an important advance in order to understand the effect of random influences on deterministic evolutionary differential equations.
\par The robustness of exponential dichotomy for nonautonomous random dynamical systems is a fundamental property in the study of stability results for random dynamics. We were able to study hyperbolicity for nonautonomous random differential equations and obtained global solutions that behave as hyperbolic equilibria. 
As in \cite{Bortolan-Carvalho-Langa,Carvalho-Langa-2,Carvalho-Langa-Robinson} this is an important step in order to obtain continuity and structure stability of attractors on nonautonomous deterministic contexts. 
\par Hence, to understand structural stability of global attractors under random perturbations, next step is to prove the existence and continuity of invariant manifolds and study lower semi-continuity of attractors for nonautonomous random dynamical systems. This will be pursuit in a forthcoming paper. 
\par The results of these paper can also by applied for general non-compact random dynamical systems, see \cite{Bixiang-Wang-existence} for a formal definition. The results presented in this work have opened the discussion of: exponential dichotomies for co-cycles driving by noncompact symbol spaces;
bounded noises in order to preserve hyperbolic structure; and possibly also continuity and structure stability of attractors under random perturbations. 

\section*{Acknowledgments}
\par This work was carried out while Alexandre Oliveira-Sousa visited the Dpto. Ecuaciones Diferenciales y An\'alisis Num\'erico (EDAN), Universidad de Sevilla and he wants to acknowledge the warm reception from people of EDAN.
We acknowledge the financial support from the following institutions: T. Caraballo and J. A. Langa by Ministerio de Ciencia, Innovaci\'on y Universidades (Spain), FEDER (European Community) under grant PGC2018-096540-B-I00, and by Proyecto I+D+i Programa Operativo FEDER Andaluc\'{\i}a US-1254251; A. N. Carvalho by S\~ao Paulo Research Foundation (FAPESP) grant 2018/10997-6, CNPq grant 306213/2019-2, and FEDER - Andalucía P18-FR-4509; and A. Oliveira-Sousa by S\~ao Paulo Research Foundation (FAPESP) grants 2017/21729-0 and 2018/10633-4, and CAPES grant PROEX-9430931/D.

\bibliographystyle{abbrv}
\bibliography{references}

\end{document}